\documentclass[oneside,english]{amsart}
\usepackage{mathptmx}

\usepackage[T1]{fontenc}
\usepackage[latin9]{inputenc}
\usepackage{geometry}
\usepackage{color}
\usepackage{babel}
\usepackage{amstext}
\usepackage{amsthm}
\usepackage{amssymb}
\usepackage[all]{xy}
\usepackage[unicode=true,pdfusetitle,
 bookmarks=true,bookmarksnumbered=false,bookmarksopen=false,
 breaklinks=false,pdfborder={0 0 0},pdfborderstyle={},backref=false,colorlinks=true]
 {hyperref}

\makeatletter
\numberwithin{equation}{section}
\numberwithin{figure}{section}
\theoremstyle{remark}
\newtheorem*{acknowledgement*}{\protect\acknowledgementname}
\theoremstyle{remark}
\newtheorem*{notation*}{\protect\notationname}
\theoremstyle{plain}
\newtheorem{thm}{\protect\theoremname}
\theoremstyle{remark}
\newtheorem{notation}[thm]{\protect\notationname}
\theoremstyle{plain}
\newtheorem{prop}[thm]{\protect\propositionname}
\theoremstyle{definition}
\newtheorem{defn}[thm]{\protect\definitionname}
\theoremstyle{plain}
\newtheorem{lem}[thm]{\protect\lemmaname}
\theoremstyle{definition}
\newtheorem{example}[thm]{\protect\examplename}
\theoremstyle{remark}
\newtheorem{rem}[thm]{\protect\remarkname}
\theoremstyle{plain}
\newtheorem{cor}[thm]{\protect\corollaryname}
\theoremstyle{plain}
\newtheorem{question}[thm]{\protect\questionname}

\usepackage{amsfonts}
\usepackage{xypic}

\makeatother

\providecommand{\acknowledgementname}{Acknowledgement}
\providecommand{\corollaryname}{Corollary}
\providecommand{\definitionname}{Definition}
\providecommand{\examplename}{Example}
\providecommand{\lemmaname}{Lemma}
\providecommand{\notationname}{Notation}
\providecommand{\propositionname}{Proposition}
\providecommand{\questionname}{Question}
\providecommand{\remarkname}{Remark}
\providecommand{\theoremname}{Theorem}

\begin{document}
\newtheoremstyle{named}{}{}{\itshape}{}{\bfseries}{.}{.5em}{\thmnote{#3}} \theoremstyle{named}
\theoremstyle{named} 
\newtheorem*{namedtheorem}{Theorem}

\author{Adrien Dubouloz}
\address{Institut de Math\'ematiques de Bourgogne, UMR 5584 CNRS, Universit\'e de Bourgogne, F-21000, Dijon \newline\indent Laboratoire de Math\'ematique et Applications, UMR 7348 CNRS, Universit\'e de Poitiers, F-86000, Poitiers} 
\email{adrien.dubouloz@maths.cnrs.fr }
\author{Takashi Kishimoto}
\address{Department of Mathematics, Faculty of Science, Saitama University, Saitama 338-8570, Japan} \email{tkishimo@rimath.saitama-u.ac.jp}
\author{Masaru Nagaoka}
\address{Gakushuin University, 1-5-1 Mejiro, Toshima-ku, Tokyo 171-8588, Japan} \email{masaru.nagaoka@gakushuin.ac.jp} 

\subjclass[2020]{14R10; 14E30; 14J30}
\keywords{affine space, completions, del Pezzo fibrations, Minimal Model Program.}
\title{Completions of the affine $3$-space into del Pezzo fibrations}
\begin{abstract}
We give constructions of completions of the affine $3$-space into
total spaces of del Pezzo fibrations of every degree other than $7$
over the projective line. We show in particular that every del Pezzo
surface other than $\mathbb{P}^{2}$ blown-up in one or two points
can appear as a closed fiber of a del Pezzo fibration $\pi:X\to\mathbb{P}^{1}$
whose total space $X$ is a $\mathbb{Q}$-factorial threefold with
terminal singularities which contains $\mathbb{A}^{3}$ as the complement
of the union of a closed fiber of $\pi$ and a prime divisor $B_{h}$
horizontal for $\pi$. For such completions, we also give a complete
description of integral curves that can appear as general fibers of
the induced morphism $\bar{\pi}:B_{h}\to\mathbb{P}^{1}$. 
\end{abstract}

\maketitle

\section*{Introduction}

Fano threefolds of small Picard rank that can appear as completions
of the affine space $\mathbb{A}^{3}$ over an algebraically closed
field of characteristic zero have received a lot of attention during
the past decades leading to a complete classification for smooth Fano
threefolds of Picard rank one \cite{Fu93} and series of partial results
for mildly singular Fano threefolds of higher Picard rank, see e.g.
\cite{HuMo20, Kis05, Nag18,Pro16}. From the viewpoint of the Minimal
Model Program (MMP), it is natural to consider more generally the
question of which threefold Mori fiber spaces $\pi:X\to B$, where
$X$ is a $\mathbb{Q}$-factorial projective threefold with terminal
singularities and where $\pi$ is an extremal contraction of relative
Picard rank one to a base variety $B$ which is either a curve or
surface, have the property to contain $\mathbb{A}^{3}$ as a Zariski
open subset of their total spaces. In the situation where the base
$B$ is a curve, hence is isomorphic to the projective line $\mathbb{P}^{1}$,
the general fibers of $\pi$ are del Pezzo surfaces of a certain degree
$d\in\{1,\ldots,9\}$ and a first basic question is to understand
which del Pezzo surfaces can appear as general fibers of del Pezzo
fibrations $\pi:X\to\mathbb{P}^{1}$ whose total spaces are completions
of $\mathbb{A}^{3}$. Due to the condition that $\pi:X\to\mathbb{P}^{1}$
has relative Picard rank one, a general fiber cannot be isomorphic
to $\mathbb{P}^{2}$ blown-up in one or two points, implying that
$d\neq7$ and that in the case $d=8$, a general fiber of $\pi$ is
isomorphic to a smooth quadric surface in $\mathbb{P}^{3}$. Besides
the obvious locally trivial $\mathbb{P}^{2}$-bundles over $\mathbb{P}^{1}$
in degree $d=9$ and families of quadric fibrations \cite{Nag20}
in degree $d=8$, particular examples of completions in every other
possible degrees can be found disseminated in the literature, in \cite{Pro16}
for $d=6$, in \cite{DKP22} for $d=5$ and in \cite{DK17,DK18} for
$d=1,2,3,4$, but in the latter four cases, without exact indication
on the possible isomorphism types of general fibers\footnote{Recall that in contrast, del Pezzo surfaces of degrees $5$ and $6$
are unique up to isomorphism.}. Our first main result completes the picture as follows: 

\begin{namedtheorem}[Theorem A] Let $S$ be a del Pezzo surface other
than $\mathbb{P}^{2}$ blown-up in one or two points. Then there exists
a completion of $\mathbb{A}^{3}$ into the total space of a del Pezzo
fibration $\pi:X\to\mathbb{P}^{1}$ such that $S$ is isomorphic to
a closed fiber of $\pi$ and the boundary divisor $B=X\setminus\mathbb{A}^{3}$
is the union of a fiber $B_{f}$ of $\pi$ and of a prime divisor
$B_{h}$ which dominates $\mathbb{P}^{1}$. 

\end{namedtheorem} 

Given a completion of $\mathbb{A}^{3}$ into the total space of a
del Pezzo fibration $\pi:X\to\mathbb{P}^{1}$ of degree $d$ for which
the boundary $B=X\setminus\mathbb{A}^{3}$ consists, as in Theorem
A, of the union of a fiber $B_{f}$ of $\pi$ and of a prime divisor
$B_{h}$ which dominates $\mathbb{P}^{1}$, the structure of the induced
surjective morphism $\bar{\pi}:B_{h}\to\mathbb{P}^{1}$ is an interesting
invariant of the isomorphism type of $X$ as a completion of $\mathbb{A}^{3}$.
The triviality of the Picard group of $\mathbb{A}^{3}$ and the fact
that $\mathbb{A}^{3}$ has only constant invertible functions imply
that the Picard group of $X$ is freely generated by the classes of
$B_{h}$ and $B_{f}$, in particular, the Picard group of the fiber
$X_{\eta}$ of $\pi$ over the generic point $\eta$ of $\mathbb{P}^{1}$
is generated by $B_{h,\eta}=B_{h}\times_{X}X_{\eta}$. Since $X_{\eta}$
is a del Pezzo surface of degree $d$ over the function field $k(t)$
of $\mathbb{P}^{1}$, it follows that for every $d\leq6$, $B_{h,\eta}$
is the support of an integral anti-canonical divisor of $X_{\eta}$.
This implies in turn that for a general closed fiber $X_{c}$ of $\pi:X\to\mathbb{P}^{1}$,
the restriction $B_{h,c}=B_{h}\times_{X}X_{c}$ is an anti-canonical
divisor of $X_{c}$, hence is isomorphic to a certain plane cubic
curve over $k$, possibly reducible. Our second main result reads
as follows:

\begin{namedtheorem}[Theorem B]Let $C$ be an integral plane cubic
curve over $k$. Then for every $1\leq d\leq6$, there exists a completion
of $\mathbb{A}^{3}$ into the total space of a del Pezzo fibration
$\pi:X\to\mathbb{P}^{1}$ of degree $d$ with boundary $X\setminus\mathbb{A}^{3}=B_{h}\cup B_{f}$
such that $C$ appears as a closed fiber of the induced morphism $\bar{\pi}:B_{h}\setminus(B_{h}\cap B_{f})\to\mathbb{P}^{1}\setminus\pi(B_{f})$.

\end{namedtheorem}

Theorem B says in particular that for every $d\in\{1,\ldots,6\}$,\textbf{
}every elliptic curve can be realized as a closed fiber of the induced
fibration $\bar{\pi}:B_{h}\to\mathbb{P}^{1}$ on a suitably chosen
completion of $\mathbb{A}^{3}$ into the total space of a del Pezzo
fibration $\pi:X\to\mathbb{P}^{1}$ of degree $d$. \\

Our general approach to prove Theorem A and Theorem B builds on techniques
initiated in \cite{DK17,DK18} and developed further in higher dimension
in \cite{DKP22} which consist of finding suitable pencils of del
Pezzo surfaces on known completions of $\mathbb{A}^{3}$ into $\mathbb{Q}$-Fano
threefolds of Picard rank one with terminal singularities and running
relative MMP's from appropriate resolution of indeterminacy of these
pencils. For a given pencil, the outputs of these relative MMP's depend
not only on the chosen resolution and of choices made at each step
of the process, but there is also no reason in general that these
steps preserve an initially prescribed open subset isomorphic to $\mathbb{A}^{3}$.
The main point of these techniques thus lies in the possibility to
find pencils and appropriate classes of resolutions for which these
MMP runs can be controlled effectively enough to ascertain that their
outputs are del Pezzo fibrations whose total spaces contain $\mathbb{A}^{3}$
as an open subset. Section 1 reviews a general framework for these
types of constructions, largely inspired from \cite{DK17,DK18,DKP22}.
In Section 2, we illustrate the use of these techniques for the construction
of completions of $\mathbb{A}^{3}$ into total spaces of del Pezzo
fibrations of every degree other that $6$ and $7$, with a particular
focus on the isomorphism types of fibers of these del Pezzo fibrations
and the structure of the boundary divisors of the completions. Section
3 is specifically devoted to the case $d=6$ which, due in particular
to our lack of knowledge of completions of $\mathbb{A}^{3}$ into
$\mathbb{Q}$-Fano threefolds of Picard rank one carrying suitable
pencils of del Pezzo surfaces of degree $6$ that fits into our general
framework, is treated by a different approach involving the construction
of explicit Sarkisov links between certain quadric fibrations and
del Pezzo fibrations of degree $6$ over $\mathbb{P}^{1}$. 
\begin{acknowledgement*}
The authors thank the Institut de Math\'{e}matiques de Bourgogne at which
the present research was initiated during visits of the second and
the third authors and Saitama University at which it was continued
during a visit of the first and the third authors for their generous
support and the excellent working conditions offered. The first author
was partially supported by the French ANR Project FIBALGA ANR-18-CE40-0003-01.
The second author was partially funded by JSPS KAKENHI Grant Number
19K03395 and 23K03047. The third author was partially supported by
JSPS KAKENHI Grant Number JP21K13768. 
\end{acknowledgement*}

\section{Del Pezzo fibration completions from del Pezzo surfaces pencils }

In this section, we review a general strategy to construct del Pezzo
fibration completions of certain quasi-projective threefolds from
suitable pencils of del Pezzo surfaces on Fano threefolds of divisor
class rank one with terminal singularities. 
\begin{notation*}
Unless otherwise explicitly stated, all varieties and schemes considered
are defined over a fixed algebraically closed field $k$ of characteristic
zero. For varieties and their singularities in the framework of the
Minimal Model Program, we follow the standard conventions and terminology
in \cite{Kol13,KM98}. 
\end{notation*}

\subsection{Pencils of divisors and their resolutions}

A \emph{pencil of divisors} on a normal projective variety $X$ is
a dominant rational map $\psi:X\dashrightarrow\mathbb{P}^{1}$. Since
$X$ is normal, such a pencil is determined by an equivalence class
of pairs $(\mathcal{F},V)$, where $\mathcal{F}$ is a coherent reflexive
sheaf of rank $1$, isomorphic to $\mathcal{O}_{X}(D)$ for some Weil
divisor $D$ on $X$, and where $V\subset H^{0}(X,\mathcal{F})$ is
a $2$-dimensional $k$-vector subspace with the property that the
support of the cokernel of the canonical evaluation homomorphism $e:V\otimes_{k}\mathcal{O}_{X}\to\mathcal{F}$
has codimension $\geq2$ in $X$. Two pairs $(\mathcal{F},V)$ and
$(\mathcal{F}',V')$ are called equivalent if and only if there exists
an isomorphism of $\mathcal{O}_{X}$-modules $\alpha:\mathcal{F}'\to\mathcal{F}$
such that $H^{0}(\alpha)(V')=V$. 

For every closed point $c$ of $\mathbb{P}^{1}$, we denote by $D_{c}\subset X$
the scheme-theoretic closure in $X$ of the scheme-theoretic fiber
over $c$ of the restriction $\psi|_{U_{\psi}}:U_{\psi}\to\mathbb{P}^{1}$
of $\psi$ to its domain of definition $U_{\psi}$. Since $\mathrm{codim}_{X}(X\setminus U_{\psi})\geq2$,
$D_{c}$ is a Weil divisor on $X$ for which $\mathcal{O}_{X}(D_{c})\cong\mathcal{F}$.
We call the divisors $D_{c}$, where $c$ ranges through the set of
closed points of $\mathbb{P}^{1}$, the \emph{members} of the pencil
$\psi$. The\emph{ base scheme} $\mathrm{Bs}\psi$ of a pencil of
divisors $\psi:X\dashrightarrow\mathbb{P}^{1}$ is the scheme-theoretic
intersection in $X$ of all its members. It is a closed sub-scheme
of $X$ whose associated reduced scheme equals $X\setminus U_{\psi}$
endowed with its reduced scheme structure. 
\begin{notation}
Given a pair of linearly equivalent Weil divisors $D$ and $D'$ on
$X$ without common irreducible components, the \emph{pencil generated
by $D$ and $D'$} is the pencil of divisors $\psi_{\langle D,D'\rangle}:X\dashrightarrow\mathbb{P}^{1}$
on $X$, unique up to an isomorphism, which has $D$ and $D'$ among
its members. Its base scheme is equal to the scheme-theoretic intersection
of $D$ and $D'$ in $X$. 
\end{notation}

A \emph{resolution} of a pencil of divisors $\psi:X\dashrightarrow\mathbb{P}^{1}$
is a birational morphism $\tau:X'\to X$ from a projective variety
$X'$ such that $\psi\circ\tau$ is a morphism. The \emph{graph} of
a pencil of divisors $\psi:X\dashrightarrow\mathbb{P}^{1}$ is the
scheme-theoretic closure $\Gamma$ in $X\times\mathbb{P}^{1}$ of
the graph of the restriction of $\psi$ to its domain of definition.
We denote by $\gamma:\Gamma\to X$ and $\mathrm{q}:\Gamma\to\mathbb{P}^{1}$
the restrictions to $\Gamma$ of the projections of $X\times\mathbb{P}^{1}$
onto its factors. The equality $\mathrm{q=}\psi\circ\gamma$ implies
that $\gamma:\Gamma\to X$ is a resolution of $\psi$, which we henceforth
call the \emph{graph resolution} of $\psi$. We denote by $E_{\Gamma}\subset\Gamma$
its exceptional locus. This resolution has the universal property
that for every resolution $\tau:X'\to X$ of $\psi$, the induced
birational map $\gamma^{-1}\circ\tau:X'\dashrightarrow\Gamma$ is
a projective morphism. The following proposition collects other properties
of the graph resolution: 
\begin{prop}[{\cite[Proposition 3.3]{DKP22}}]
\label{prop:Graph-resolution-properties} For a pencil $\psi:X\dashrightarrow\mathbb{P}^{1}$
on a normal projective variety $X$, the following hold:

a) The graph resolution $\gamma:\Gamma\to X$ restricts to an isomorphism
over $X\setminus\mathrm{Bs}\psi$ and for every closed point $x$
of $\mathrm{Bs}\psi$, $\gamma^{-1}(x)\subset\Gamma$ is a section
of $\mathrm{q}:\Gamma\to\mathbb{P}^{1}$. 

b) For every closed point $c\in\mathbb{P}^{1}$, the fiber $\Gamma_{c}=\mathrm{q}^{-1}(c)$
equals the proper transform $\gamma_{*}^{-1}D_{c}$ in $\Gamma$ of
the member $D_{c}$ of $\psi$ and the birational morphism $\gamma_{c}:(\Gamma_{c},(E_{\Gamma}\cap\Gamma_{c})_{\mathrm{red}})\to(D_{c},(\mathrm{Bs}\psi)_{\mathrm{red}})$
induced by $\gamma$ is an isomorphism of pairs. 
\end{prop}

Recall that a $\mathbb{Q}$-\emph{factorial terminalization} of a
normal projective variety $X$ is a proper birational morphism $f:X'\to X$
such that $X'$ is a projective $\mathbb{Q}$-factorial variety with
terminal singularities such that $K_{X'}$ is $f$-nef. Every normal
projective variety $X$ admits $\mathbb{Q}$-factorial terminalizations
\cite[Theorem 1.33 and Corollary 1.37]{Kol13}, moreover, the restriction
of any $\mathbb{Q}$-factorial terminalization $f:X'\to X$ of $X$
over every $\mathbb{Q}$-factorial open subset of $X$ with terminal
singularities is an isomorphism. The following notion was introduced
in \cite{DKP22}: 
\begin{defn}
A \emph{thrifty resolution} of a pencil of divisors $\psi:X\dashrightarrow\mathbb{P}^{1}$
on a normal projective variety $X$ is a resolution $\tau:X'\to X$
of $\psi$ where $X'$ is a $\mathbb{Q}$-factorial projective variety
with terminal singularities such that the canonically induced morphism
$\sigma:X'\to\Gamma$ is a $\mathbb{Q}$-factorial terminalization
of the normalization of $\Gamma$. 
\end{defn}

\subsection{del Pezzo fiber spaces and relative MMP's }

By a \emph{del Pezzo surface} over a field $F$ of characteristic
zero, we mean a geometrically connected surface $S$ projective and
smooth over $F$ whose anti-canonical invertible sheaf $\omega_{S/F}^{\vee}=\det\Omega_{S/F}^{\vee}$
is ample. The \emph{degree} of $S$ is the integer $d(S)=\dim H^{0}(S,\omega_{S/F}^{\vee})-1.$
Note that by the flat base change and the Kodaira vanishing theorem
one has $H^{i}(S,\mathcal{O}_{S})=H^{i}(S,\omega_{S/F}^{\vee})=0$
for $i\geq1$, and it follows in turn from the Riemman-Roch theorem
that $d(S)=\deg(c_{1}(\omega_{S/F}^{\vee}))^{2}=(-K_{S})^{2}$, where
$-K_{S}$ is any anti-canonical divisor on $S$. In what follows,
unless otherwise explicitly stated, the term del Pezzo surface always
refers to a del Pezzo surface over $k$. In the next sections, we
will freely make use of many classical results on del Pezzo surfaces,
their anti-canonical divisors and their anti-canonical models, referring
the reader for instance to \cite{Dol12,Man74} for the details. We
record the following lemma which is a straightforward consequence
of the fact that every del Pezzo surface of degree $d$ over an algebraically
closed field of characteristic zero is either isomorphic to $\mathbb{P}^{1}\times\mathbb{P}^{1}$or
to the blow-up of $\mathbb{P}^{2}$ at $9-d$ closed points in general
position. 
\begin{lem}
\label{lem:Prime-anticanonical=00003Dplane-cubic-curve}The support
of a prime anti-canonical divisor on a del Pezzo surface is isomorphic
to an integral plane cubic curve $C$. Conversely, for every degree
$d\in\{1,\ldots,9\}$ and every integral plane cubic curve $C$, there
exists a del Pezzo surface $S$ of degree $d$ and a prime anti-canonical
divisor $D$ on $S$ such that $D\cong C$. 
\end{lem}

\begin{defn}
A \emph{del Pezzo fiber space of degree} $d$ over a smooth connected
curve $C$ is an irreducible threefold $X$ endowed with a faithfully
flat projective morphism $\pi:X\to C$ whose fiber $X_{\eta}$ over
the generic point $\eta$ of $C$ is a del Pezzo surface of degree
$d$ over the field of functions $k(C)$ of $C$. A fiber of $\pi:X\to C$
is called \emph{degenerate} if it is either not an integral scheme
or is not contained in the regular locus $X_{\mathrm{reg}}$ of $X$.
We denote by $\delta_{X/C}\subset C$ the finite set of points over
which the fiber of $\pi$ is degenerate and by $V_{X/C}\subseteq X_{\mathrm{reg}}$
the open complement of $\pi^{-1}(\delta_{X/C})$ in $X$. 

A \emph{del Pezzo fibration} is a del Pezzo fiber space $\pi:X\to C$
of relative Picard rank one such that $X$ is $\mathbb{Q}$-factorial
with terminal singularities and $-K_{X}$ is $\pi$-ample. 
\end{defn}

\begin{lem}
\label{lem:delPezzo-fibrebundle-properties}Let $C$ be a smooth connected
curve, let $X$ be an irreducible threefold and let $d\in\{1,\ldots,9\}$
be an integer. For a faithfully flat projective morphism $\pi:X\to C$,
the following properties are equivalent: 

a) The morphism $\pi:X\to C$ is a del Pezzo fiber space of degree
$d$. 

b) There exists a closed point $c\in C$ such that the fiber $X_{c}=\pi^{-1}(c)$
is a del Pezzo surface of degree $d$.

c) A general closed fiber of $\pi$ is a del Pezzo surface of degree
$d$. 
\end{lem}

\begin{proof}
Since $\pi:X\to C$ is flat and projective, the subset $W$ of $C$
consisting of points $c$ such that $\pi^{-1}(c)$ is regular and
geometrically connected is open by \cite[Th\'eor\`eme 12.2.4]{Gro66}.
Furthermore, by \cite[Proposition 17.8.2]{Gro67}, the restriction
of $\pi$ over $W$ is a smooth projective morphism with geometrically
connected fibers. Since each of the properties listed implies that
$W$ is Zariski dense, to prove the assertion, up to replacing $C$
by an affine open subset $W$, we can assume without loss of generality
that $C$ is affine and that $\pi:X\to C$ is a smooth projective
morphism with geometrically connected fibers. Up to shrinking $C$
further if necessary, we can assume as well that the canonical sheaf
$\omega_{C}$ is trivial, hence that the canonical sheaf $\omega_{X}=\Lambda^{3}\Omega_{X/k}$
equals the relative canonical sheaf $\omega_{X/C}=\Lambda^{2}\Omega_{X/C}$.
By \cite[Th\'eor\`eme 4.7.1]{Gro61}, the subset $W'$ of $C$ consisting
of points such the restriction of $\omega_{X/C}^{\vee}$ to $X_{c}$
is ample is open. Since for every $c\in C$, $\omega_{X/C}^{\vee}|_{X_{c}}\cong\omega_{X_{c}/\kappa(c)}^{\vee}$,
where $\kappa(c)$ denotes the residue field of $c$, each of the
properties listed implies that $W'$ is Zariski dense. Thus, by replacing
$C$ by $W'$, we are reduced to the situation where $\pi:X\to C$
is a smooth projective morphism of relative dimension $2$ with geometrically
connected fibers such that $\omega_{X}^{\vee}\cong\omega_{X/C}^{\vee}$
is $\pi$-ample. The vanishing of $H^{i}(X_{c},\omega_{X_{c}/\kappa(c)}^{\vee})$
for every $c\in C$ and every $i\geq1$ then implies by \cite[Corollaire 7.9.9]{Gro63}
that $\pi_{*}\omega_{X/C}^{\vee}$ is locally free, hence that all
fibers of $\pi$ are del Pezzo surfaces of the same degree $d=\mathrm{rk\pi_{*}\omega_{X/C}^{\vee}-1}$. 
\end{proof}
\begin{example}
Let $\psi:X\dashrightarrow\mathbb{P}^{1}$ be a pencil of divisors
on a normal projective threefold $X$ which has a del Pezzo surface
of degree $d$ among its members and let $\gamma:\Gamma\to X$ be
its graph resolution. Proposition \ref{prop:Graph-resolution-properties}
b) and Lemma \ref{lem:delPezzo-fibrebundle-properties} imply that
the induced morphism $\mathrm{q}=\psi\circ\gamma:\Gamma\to\mathbb{P}^{1}$
is a del Pezzo fiber space of degree $d$ whose closed fibers are
isomorphic to the members of $\psi$. In particular, a general member
of $\psi$ is a del Pezzo surface of degree $d$. 
\end{example}

Let $\pi_{Y}:Y\to C$ be a del Pezzo fiber space of degree $d$ where
$Y$ is a $\mathbb{Q}$-factorial threefold with terminal singularities.
Recall \cite[3.31]{KM98} that a MMP $\varphi:Y=Y_{0}\dashrightarrow Y{}_{n}=\tilde{Y}$
relative to $\pi{}_{Y,0}=\pi_{Y}$ is a finite sequence $\varphi=\varphi_{n}\circ\cdots\circ\varphi_{1}$
of birational maps 
\begin{eqnarray*}
Y_{\ell-1} & \stackrel{\varphi_{\ell}}{\dashrightarrow} & Y_{\ell}\\
\pi{}_{Y,\ell-1}\downarrow &  & \downarrow\pi{}_{Y,\ell}\qquad\ell=1,\ldots,n,\\
C & = & C
\end{eqnarray*}
where each $\varphi_{\ell}$ is either a divisorial contraction associated
to an extremal ray $R_{\ell-1}$ of the closure $\overline{NE}(Y{}_{\ell-1}/C)$
of the relative cone of effective curves of $Y_{\ell-1}$ over $C$,
which contracts a prime divisor either horizontal for $\pi_{Y,\ell-1}:Y_{\ell-1}\to C$
or contained in a reducible closed fiber of $\pi_{Y,\ell-1}:Y_{\ell-1}\to C$,
or a flip whose flipping and flipped curves are contained in the fibers
of $\pi{}_{Y,\ell-1}$ and $\pi{}_{Y,\ell}$ respectively, and such
that $\tilde{\pi}_{Y}=\pi{}_{Y,n}:\tilde{Y}=Y_{n}\to C$ is one of
the following:

1) A del Pezzo fibration $\tilde{\pi}:\tilde{Y}\to C$ of degree $\tilde{d}\geq d$
such that the induced map $\varphi_{\eta}:Y{}_{\eta}\dashrightarrow\tilde{Y}_{\eta}$
between the generic fibers of $\pi$ and $\tilde{\pi}$ is a morphism
consisting of the contraction of $\tilde{d}-d$ disjoint curves $E_{i}$
in $Y_{\eta}$ which satisfy $K_{Y_{\eta}}\cdot E_{i}=-1$, 

2) A Mori conic bundle $h:\tilde{Y}\to T$ over a certain normal surface
$h':T\to C$ projective over $C$. \\

A relative MMP over $C$ ran from a given del Pezzo fiber space $\pi_{Y}:Y\to C$
where $Y$ is a $\mathbb{Q}$-factorial threefold with terminal singularities
is in general not unique, and the isomorphism types as schemes over
$C$ of its possible outputs depend on choices made at each of the
steps. Nevertheless, imposing the condition that the generic fiber
of $\pi_{Y}$ has Picard rank one over $k(C)$ discards the possibility
that divisorial contractions of horizontal prime divisors for $\pi_{Y}:Y\to C$
occur in the process, leading to the following result: 
\begin{prop}
\label{prop:MMP-del-Pezzo-space-PicardRank1}Let $\pi:X\to C$ be
a del Pezzo fiber space whose generic fiber has Picard rank one over
$k(C)$. Then every MMP $\varphi:Y\dashrightarrow\tilde{Y}$ relative
to $C$ ran from a $\mathbb{Q}$-factorial terminalization $\tau:Y\to X$
of the normalization of $X$ terminates with a del Pezzo fibration
$\tilde{\pi}:\tilde{Y}\to C$. Moreover, the birational map $(\varphi\circ\tau^{-1})|_{V_{X/C}}:V_{X/C}\dashrightarrow\tilde{Y}$
is an open immersion of $C$-schemes whose image is contained in $V_{\tilde{Y}/C}$. 
\end{prop}

\begin{proof}
The induced morphism $\pi_{Y}=\pi\circ\tau:Y\to C$ is flat and projective
and since $\tau$ restricts to an isomorphism over the regular locus
of $X$, hence over $V_{X/C}$, it follows that the induced rational
map $\tau^{-1}|_{V_{X/C}}:V_{X/C}\dashrightarrow Y$ is an open immersion
of $C$-schemes whose image is contained in $V_{Y/C}$. In particular,
$\pi_{Y}:Y\to C$ is a del Pezzo fiber space whose generic fiber $Y_{\eta}$
is isomorphic to that $X_{\eta}$ of $\pi$. Since $X_{\eta}$ has
Picard rank one over $k(C)$, $\varphi_{\eta}\circ\tau_{\eta}^{-1}:X{}_{\eta}\to\tilde{Y}_{\eta}$
is necessarily an isomorphism. This implies that $\tilde{\pi}:\tilde{Y}\to C$
cannot factor through a Mori conic bundle and hence, that $\tilde{\pi}:\tilde{Y}\to C$
is a del Pezzo fibration of the same degree as that of the del Pezzo
fiber space $\pi:X\to C$. Since $\tau^{-1}|_{V_{X/C}}:V_{X/C}\to Y$
is an open immersion with image contained in $V_{Y/C}$, the second
assertion follows from the fact that $\varphi$ induces an open immersion
of $C$-schemes $V_{Y/C}\hookrightarrow\tilde{Y}$, whose image is
contained in $V_{\tilde{Y}/C}$. Indeed, by the previous observation,
the only prime divisors in $Y$ that can be contracted by $\varphi$
are irreducible components of reducible closed fibers of $\pi_{Y}$.
In particular, the restriction of $\varphi$ to $V_{Y/C}$ does not
contract any divisor, hence consists at worse of a sequence of flips.
But since $V_{Y/C}$ is contained in $Y_{\mathrm{reg}}$ and flipping
curves must pass through singular points \cite[14.5.4]{CKM88}, we
conclude by induction that each step $\varphi_{\ell}:Y_{\ell-1}\dashrightarrow Y_{\ell}$,
$\ell=1,\ldots,n$, restricts to an open immersion of $C$-schemes
$V_{Y_{\ell-1}/C}\hookrightarrow V_{Y_{\ell}/C}$, hence that $\varphi$
induces an open immersion of $C$-schemes $V_{Y/C}\hookrightarrow\tilde{Y}$. 
\end{proof}

\subsection{Special del Pezzo pencils on $\mathbb{Q}$-Fano threefolds of divisor
class rank one }

We now introduce a class of pencils of del Pezzo surfaces $\psi:X\dashrightarrow\mathbb{P}^{1}$
on suitable projective threefolds $X$ for which we can ascertain
that the output of any relative MMP ran from any thrifty resolution
$\tau:Y\to X$ of $\psi$ is a del Pezzo fibration $\tilde{\pi}:\tilde{Y}\to\mathbb{P}^{1}$
whose general closed fibers are isomorphic to general members of $\psi$. 
\begin{defn}
\label{def:Special-del-Pezzo-pencil} An\emph{ $H$-special del Pezzo
pencil} is a triple $(X,H,\psi)$ consisting of a projective threefold
$X$ of divisor class rank one with terminal singularities, an effective
prime divisor $H$ on $X$ such that $\mathrm{Cl}(X)=\mathbb{Z}[H]$
and a pencil of divisors $\psi:X\dashrightarrow\mathbb{P}^{1}$ which
satisfies the following properties:

a) $\psi$ has a member which is a del Pezzo surface,

b) $mH$ is a member of $\psi$ for some integer $m\geq1$, 

c) The base scheme $\mathrm{Bs}\psi$ of $\psi$ is irreducible,

d) If $m=1$ then $\mathrm{Bs}\psi$ is reduced. 
\end{defn}

\begin{rem}
\label{rem:DelPezzo-Pencil-Consequences}Note that a projective threefold
$X$ of divisor class rank one is automatically $\mathbb{Q}$-factorial
as the image of the natural inclusion $\mathrm{Pic}(X)\hookrightarrow\mathrm{Cl}(X)$
is a nontrivial subgroup of finite index. The existence of an $H$-special
del Pezzo pencil $\psi:X\dasharrow\mathbb{P}^{1}$ implies the following
additional properties: 

1) First, the hypotheses that $X$ has terminal singularities and
that $\psi$ has a regular member imply by a known index one cover
argument \cite[Lemma 5.3 (2) and Proposition 6.7 (2)]{Kol92} that
every member of $\psi$ is a Cartier divisor on $X$. 

2) Second, by the adjunction formula for a general member of $\psi$,
$X$ is necessarily a Fano threefold of Fano index strictly bigger
than the integer $m$ for which $mH$ is a member of $\psi$. 
\end{rem}

Given an $H$-special del Pezzo pencil $(X,H,\psi)$ with graph resolution
$\gamma:\Gamma\to X$, Proposition \ref{prop:Graph-resolution-properties}
b) and Lemma \ref{lem:delPezzo-fibrebundle-properties} imply that
$\mathrm{q=\psi\circ\gamma:\Gamma\to\mathbb{P}^{1}}$ is a del Pezzo
fiber space. The following proposition provides a measure of how much
$\gamma:\Gamma\to X$ and $\mathrm{q}:\Gamma\to\mathbb{P}^{1}$ differ
in general from a thrifty resolution of $\psi$ and a del Pezzo fibration,
respectively. 
\begin{prop}
\label{prop:Graph-Special-delPezzo-pencil}Let $(X,H,\psi)$ be an
$H$-special del Pezzo pencil and let $\gamma:\Gamma\to X$ be the
graph resolution of $\psi$. Then the following hold: 

a) Every member of $\psi$ other than $mH$ is a prime divisor; consequently,
every fiber of $\mathrm{q}:\Gamma\to\mathbb{P}^{1}$ other than $\mathrm{q}^{-1}(q(\gamma_{*}^{-1}H))$
is integral. 

b) The variety $\Gamma$ is $\mathbb{Q}$-factorial with class group
generated by the classes of the exceptional locus $E_{\Gamma}$ of
$\gamma$ and of the proper transform $\gamma_{*}^{-1}H$ of $H$. 

c) The open subset $\Gamma\setminus(E_{\Gamma}\cap\gamma_{*}^{-1}H)$
has terminal singularities and the following additional properties
hold:

$\quad$ (i) If $m\geq2$ then $\Gamma\setminus\gamma_{*}^{-1}H$
is regular in a neighborhood of $E_{\Gamma}\setminus(E_{\Gamma}\cap\gamma_{*}^{-1}H)$,

$\quad$ (ii) If $m=1$ then $\Gamma$ has terminal singularities. 
\end{prop}

\begin{proof}
Let $D$ be a regular member of $\psi$, let $x\in\mathrm{Bs}\psi$
be a closed point and let $\mathfrak{m}$ be the maximal ideal of
$\mathcal{O}_{X,x}$. Since by Remark \ref{rem:DelPezzo-Pencil-Consequences}
1), the members of $\psi$ are Cartier, the ideals of $D$ and $mH$
in $\mathcal{O}_{X,x}$ are principal, say respectively generated
by elements $f,h\in\mathfrak{m}.$ Moreover, since $D$ is regular
at $x$, $\mathcal{O}_{X,x}/(f)$ is a regular ring, which implies
that $\mathcal{O}_{X,x}$ is regular and that $f\in\mathfrak{m}\setminus\mathfrak{m}^{2}$.
If $m=1$ then since $H$ generates $\mathrm{Cl}(X)$, it follows
that every member of $\psi$ is a prime divisor. Now assume that $m\geq2$.
Then $h\in\mathfrak{m}^{2}$ and the ideal $I_{D',x}$ in $\mathcal{O}_{X,x}$
of a member $D'$ of $\psi$ other than $mH$ is generated by an element
of the form $f'=f+th$ for some $t\in k$, which thus belongs to $\mathfrak{m}\setminus\mathfrak{m}^{2}$.
Since every irreducible component $D_{i}'$ of $D'$ is $\mathbb{Q}$-Cartier
and ample, the support of $D_{i}'\cap D$ is a closed subset of pure
codimension $2$ of $X$ contained in $\mathrm{Supp}(\mathrm{Bs}\psi)$,
hence equal to it since $\mathrm{Bs}\psi$ is irreducible. It follows
that the ideal $I_{D'_{i},x}$ in $\mathcal{O}_{X,x}$ of each irreducible
component $D_{i}'$ is contained in $\mathfrak{m}$, hence, since
$I_{D',x}$ is contained in $\mathfrak{m}$ but not in $\mathfrak{m}^{2}$,
that $D'$ is a prime divisor, regular at $x$. Assertion a) then
follows from Proposition \ref{prop:Graph-resolution-properties} b)
which asserts that $\mathrm{q}^{-1}(\mathrm{q}(\gamma_{*}^{-1}D'))=\gamma_{*}^{-1}D'$
is isomorphic to $D'$ and is a prime Cartier divisor. 

Since $\mathrm{Bs}\psi$ is the scheme-theoretic intersection of any
two general members of $\psi$ and the latter are Cartier by Remark
\ref{rem:DelPezzo-Pencil-Consequences} 1), the ideal sheaf $\mathcal{I}_{\mathrm{Bs}\psi}\subset\mathcal{O}_{X}$
of $\mathrm{Bs}\psi$ is locally generated by a regular sequence of
length $2$. This implies, see e.g. \cite[Proposition 7.1.2 and \S 7.1.3]{Dol12},
that $\gamma:\Gamma\to X$ equals the blow-up of $X$ along $\mathrm{Bs}\psi$,
hence that $E_{\Gamma}=(\gamma^{-1}(\mathrm{Bs}\psi))_{\mathrm{red}}$
is a prime divisor which is the support of a Cartier divisor on $\Gamma$.
On the other hand, since $\mathrm{Bs}\psi$ has codimension $2$ in
$X$, the restriction homomorphism $\mathrm{\mathrm{Cl}(X)\to\mathrm{Cl}(X\setminus\mathrm{Bs}\psi)}\cong\mathrm{Cl}(\Gamma\setminus E_{\Gamma})$
is an isomorphism and hence, since $\mathrm{Cl}(X)=\mathbb{Z}[H]$,
the class group of $\Gamma$ is generated by $E_{\Gamma}$ and $\gamma_{*}^{-1}H$.
Assertion b) then follows from the observation that for a general
member $D$ of $\psi$, one has $m\gamma_{*}^{-1}H\sim\gamma_{*}^{-1}D=\mathrm{q}^{-1}(q(\gamma_{*}^{-1}D))$
which is a prime Cartier divisor on $\Gamma$. 

To prove Assertion c), we first note that $\Gamma\setminus E_{\Gamma}\cong X\setminus\mathrm{Bs}\psi$
has terminal singularities by assumption. If $m\geq2$, then we established
above in passing that every member $D'$ of $\psi$ other than $mH$
is a prime divisor regular at every point $x$ of $\mathrm{Bs}\psi$.
Since $\mathrm{q}^{-1}(\mathrm{q}(\gamma_{*}^{-1}D'))=\gamma_{*}^{-1}D'\cong D'$
is a prime Cartier divisor on $\Gamma$ which is regular at every
point of $\mathrm{q}^{-1}(\mathrm{q}(\gamma_{*}^{-1}D'))\cap E_{\Gamma}$,
it follows that $\Gamma$ is regular in a neighborhood of $\mathrm{q}^{-1}(\mathrm{q}(\gamma_{*}^{-1}D'))\cap E_{\Gamma}$,
and hence that $\Gamma\setminus\gamma_{*}^{-1}H$ is regular in a
neighborhood of $E_{\Gamma}\setminus(E_{\Gamma}\cap\gamma_{*}^{-1}H)$.
Now assume that $m=1$ and let $D'$ be a regular member of $\psi$
other than $D$. For every closed point $x\in\mathrm{Bs}\psi$, the
ideal of $D'$ in $\mathcal{O}_{X,x}$ is principal, generated by
an element $f'\in\mathfrak{m}\setminus\mathfrak{m}^{2}$ and the ideal
$\mathcal{I}_{\mathrm{Bs}\psi,x}$ of $\mathrm{Bs}\psi$ in $\mathcal{O}_{X,x}$
is generated by $f$ and $f'$. Since $\mathcal{O}_{X,x}$ and $\mathcal{O}_{X,x}/(f)$
are regular, there exist elements $g_{1},g_{2}\in\mathfrak{m}$ such
that the residue classes of $f,g_{1},g_{2}$ modulo $\mathfrak{m}^{2}$
form a basis of $\mathfrak{m}/\mathfrak{m}^{2}$. The formal completion
$\hat{\mathcal{O}}_{X,x}$ of $\mathcal{O}_{X,x}$ with respect to
$\mathfrak{m}$ is then isomorphic to the formal power series ring
$k[[\hat{f},y,z]]$, generated by the respective images of $f$, $g_{1}$
and $g_{2}$. Since $\mathrm{Bs}\psi$ is reduced, the image of $f'$
in $\hat{\mathcal{O}}_{X,x}$ is, up to multiplication by an invertible
element, of the form $\hat{f}'=p_{0}(y,z)+\hat{f}r(\hat{f},y,z)$
where $p_{0}$ is a nonconstant polynomial without square factors
and $r$ is a formal power series. The image in $\hat{\mathcal{O}}_{X,x}$
of $\mathcal{I}_{\mathrm{Bs}\psi,x}$ equals the ideal $\hat{I}=(\hat{f},p_{0}(y,z)+\hat{f}r(\hat{f},y,z))=(\hat{f},p_{0}(y,z))$
whose associated Rees algebra $\bigoplus_{m\geq0}\hat{I}^{m}$ is
isomorphic to $k[[\hat{f},y,z]][u,v]/(\hat{f}v-p_{0}(y,z)u)$. Letting
$n=\mathrm{mult}_{(0,0)}p_{0}$, we conclude that if $n=1$ then $\Gamma$
is regular along $\gamma^{-1}(x)$. Otherwise, if $n\geq2$, then
$\Gamma$ has a unique compound Du Val singularity of type $cA_{n-1}$,
in particular a terminal singularity, supported on $\gamma^{-1}(x)$.
Noting that $\mathrm{Bs}\psi$ is generically regular because it is
reduced and that, in the description above, $\mathrm{mult}_{(0,0)}p_{0}>1$
if and only if $x$ is a singular point of $\mathrm{Bs}\psi$ yields
the conclusion that $\Gamma$ has at worse finitely many cDV singularities
supported on $E_{\Gamma}$, which completes the proof of Assertion
c) (ii).
\end{proof}
\begin{cor}
\label{cor:Thrifty-to-graph}Let $(X,H,\psi)$ be an $H$-special
del Pezzo pencil, let $\gamma:\Gamma\to X$ be the graph resolution
of $\psi$, let $\tau:Y\to X$ be a thrifty resolution of $\psi$
and let $\sigma:Y\to\Gamma$ be the induced morphism. Then the following
hold: 

a) If $m=1$ then $\sigma$ is an isomorphism and $\mathrm{q}\circ\sigma:\Gamma\to\mathbb{P}^{1}$
is a del Pezzo fibration whose closed fibers are isomorphic to the
members of $\psi$. 

b) If $m\geq2$ then $\sigma$ restricts to an isomorphism over $\Gamma\setminus\gamma_{*}^{-1}H$. 
\end{cor}

\begin{proof}
If $m=1$ then by Proposition \ref{prop:Graph-Special-delPezzo-pencil}
b) and c) (ii), $\Gamma$ is $\mathbb{Q}$-factorial and has terminal
singularities, which implies that $\sigma:Y\to\Gamma$ is an isomorphism.
Moreover, since $m=1$, $\mathrm{q:}\Gamma\to\mathbb{P}^{1}$ has
integral closed fibers, the exceptional locus $E_{\Gamma}$ of $\gamma$
is a prime Cartier divisor and $\mathrm{Pic}(\Gamma)$ is generated
by the classes of $E_{\Gamma}$ and $\gamma_{*}^{-1}H$. Since for
every closed point $c\in\mathbb{P}^{1}$, $\gamma$ maps the pair
$(\Gamma_{c},E_{\Gamma}\cap\Gamma_{c})$ isomorphically onto the pair
$(\gamma(\Gamma_{c}),\mathrm{Bs}\psi)$ on which $\mathrm{Bs}\psi$
is ample, it follows that $E_{\Gamma}$ is $\mathrm{q}$-ample. Writing
$-K_{\Gamma}\sim aE_{\Gamma}+b\gamma_{*}^{-1}H$, we have $a>0$ because
a general closed fiber of $\mathrm{q}$ is a del Pezzo surface, which
implies that $-K_{\Gamma}$ is $\mathrm{q}$-ample and completes the
proof of Assertion a). Assertion b) immediately follows from Proposition
\ref{prop:Graph-Special-delPezzo-pencil} c) (i). 
\end{proof}
\begin{cor}
\label{cor:delPezzo-completion-from-pencil} Let $(X,H,\psi)$ be
an $H$-special del Pezzo pencil. Then every relative MMP $\varphi:Y\dashrightarrow\tilde{Y}$
over $\mathbb{P}^{1}$ ran from a thrifty resolution $\tau:Y\to X$
of $\psi$ terminates with a del Pezzo fibration $\tilde{\pi}:\tilde{Y}\to\mathbb{P}^{1}$.
Moreover, the induced map 
\[
(\varphi\circ\sigma^{-1})|_{V_{\Gamma/\mathbb{P}^{1}}}:(V_{\Gamma/\mathbb{P}^{1}},E_{\Gamma}|_{V_{\Gamma/\mathbb{P}^{1}}})\dashrightarrow(\tilde{Y},\varphi_{*}\sigma_{*}^{-1}E_{\Gamma}),
\]
where $\sigma:Y\to\Gamma$ is the induced morphism, is an open immersion
of pairs over $\mathbb{P}^{1}$. 
\end{cor}

\begin{proof}
If $m=1$, the assertion follows from Corollary \ref{cor:Thrifty-to-graph}
a). So assume that $m\geq2$. Since $\mathrm{Cl}(\Gamma)$ is generated
by the classes of the exceptional locus $E_{\Gamma}$ of $\gamma:\Gamma\to X$
and of $\gamma_{*}^{-1}H$, the generic fiber $\Gamma_{\eta}$ of
the del Pezzo fiber space $\mathrm{q}:\Gamma\to\mathbb{P}^{1}$ has
Picard rank one and the assertion follows from Proposition \ref{prop:MMP-del-Pezzo-space-PicardRank1}. 
\end{proof}
Specializing further to $H$-special pencils for which $X\setminus H$
is regular, we obtain the following:
\begin{thm}
\label{thm:Completion}Let $(X,H,\psi)$ be an $H$-special del Pezzo
pencil. Assume that $m=1$ or that $m\geq2$ and $X\setminus H$ is
regular. Then every relative MMP $\varphi:Y\dashrightarrow\tilde{Y}$
over $\mathbb{P}^{1}$ ran from a thrifty resolution $\tau:Y\to X$
of $\psi$ terminates with a del Pezzo fibration $\tilde{\pi}:\tilde{Y}\to\mathbb{P}^{1}$
whose total space is a completion of $X\setminus H$ with boundary
divisor 
\[
B=B_{h}\cup B_{f}=\varphi_{*}\sigma_{*}^{-1}E_{\Gamma}\cup\varphi_{*}(\sigma^{-1}(\gamma_{*}^{-1}H)),
\]
where $\sigma:Y\to\Gamma$ is the induced morphism. Moreover, for
every member $D_{c}$ of $\psi$ other than $mH$, the composition
$\varphi\circ\tau{}^{-1}|_{D_{c}}:(D_{c},(\mathrm{Bs}\psi)_{\mathrm{red}})\to(\tilde{Y}_{c,}B_{h,c})$
is an isomorphism of pairs. 
\end{thm}

\begin{proof}
If $m=1$, then, by Corollary \ref{cor:Thrifty-to-graph} a), $\tilde{\pi}:\tilde{Y}\to\mathbb{P}^{1}$
is isomorphic over $\mathbb{P}^{1}$ to $\mathrm{q}:\Gamma\to\mathbb{P}^{1}$.
Otherwise, if $m\geq2$ then $V_{\Gamma/\mathbb{P}^{1}}=\Gamma\setminus\gamma_{*}^{-1}H$
by Proposition \ref{prop:Graph-Special-delPezzo-pencil} a) and c)
(i) and 
\[
\varphi\circ\sigma^{-1}:(\Gamma\setminus\gamma_{*}^{-1}H,E_{\Gamma}|_{\Gamma\setminus\gamma_{*}^{-1}H})\to(\tilde{Y}\setminus B_{f},B_{h}|_{\tilde{Y}\setminus B_{f}})
\]
is an isomorphism of pairs over $\mathbb{P}^{1}$ by Corollary \ref{cor:delPezzo-completion-from-pencil}.
In both cases, the assertion then follows from the fact that $\gamma^{-1}:X\setminus H\to\Gamma$
is an open immersion as the complement of $E_{\Gamma}\cup\gamma_{*}^{-1}H$
and Proposition \ref{prop:Graph-resolution-properties} b). 
\end{proof}

\section{Applications to the construction of completions of the affine $3$-space }

We now apply the results of the previous section to the construction
of completions of the affine space $\mathbb{A}^{3}$ into total spaces
of del Pezzo fibrations over $\mathbb{P}^{1}$. The common principle
of the constructions below is to seek for $H$-special del Pezzo pencils
$(X,H,\psi)$ on known completions of $\mathbb{A}^{3}$ into Fano
threefolds $X$ of divisor class rank one with terminal singularities
having the property that the open subset $X\setminus H$ is isomorphic
to $\mathbb{A}^{3}$ and to apply Theorem \ref{thm:Completion}. In
every degree $d\in\{1,\ldots,9\}$ other than $d=6,7$ we are able
to provide with this approach completions of $\mathbb{A}^{3}$ into
total spaces of del Pezzo fibrations of degree $d$ over $\mathbb{P}^{1}$.
We do not claim the novelty of many of the constructions and examples
listed below: most of these are similar to already existing ones disseminated
in the literature. Our main point here is rather to review and collect
these in a more uniform and accurate way within the framework of $H$-special
del Pezzo pencils, with a particular focus on the isomorphism types
of fibers of these del Pezzo fibrations and the structure of the boundary
divisors of the completions. 

\subsection{del Pezzo fibration completions from special pencils on weighted
projective spaces}

Here we collect classes of examples of del Pezzo fibration completions
of $\mathbb{A}^{3}$ of degree $d\in\{1,2,3,8,9\}$ obtained from
suitable special del Pezzo pencils on completions of $\mathbb{A}^{3}$
into weighted projective spaces. The cases $d=8,9$ are very classical,
and for $d=1,2,3$, the examples we present are essentially more precise
reinterpretations of those already constructed earlier in \cite{DK17}
by similar methods in a slightly different language. 

\subsubsection{Examples obtained from special pencils on the projective space $\mathbb{P}^{3}$ }

In the next examples, we view $\mathbb{A}^{3}$ as the complement
of a hyperplane $H\subset\mathbb{P}^{3}$. \\

$\bullet$ Let $H'\subset\mathbb{P}^{3}$ be a hyperplane other than
$H$ and let $\psi_{9}:\mathbb{P}^{3}\dashrightarrow\mathbb{P}^{1}$
be the pencil generated by $H$ and $H'$. Then $(\mathbb{P}^{3},H,\psi_{9})$
is an $H$-special del Pezzo pencil with $m=1$ which gives rise to
a completion of $\mathbb{A}^{3}=\mathbb{P}^{3}\setminus H$ into the
total space of a del Pezzo fibration $\mathrm{q}:\Gamma\to\mathbb{P}^{1}$
of degree $9$ with boundary divisor $B=B_{h}\cup B_{f}=E_{\Gamma}\cup\gamma_{*}^{-1}H$,
where the graph resolution $\gamma:\Gamma\to\mathbb{P}^{3}$ is the
blow-up of $\mathbb{P}^{3}$ along the line $L=H\cap H'$ with exceptional
divisor $E_{\Gamma}\cong\mathbb{P}(\mathcal{O}_{\mathbb{P}^{1}}(-1)\oplus\mathcal{O}_{\mathbb{P}^{1}}(-1))$
and $\mathrm{q}:\Gamma\to\mathbb{P}^{1}$ is the projective bundle
$\mathbb{P}(\mathcal{O}_{\mathbb{P}^{1}}\oplus\mathcal{O}_{\mathbb{P}^{1}}(-1)^{\oplus2})\to\mathbb{P}^{1}$.
Here, the restriction $\bar{\mathrm{q}}:B_{h}\to\mathbb{P}^{1}$ is
the trivial $\mathbb{P}^{1}$-bundle. \\

$\bullet$ Let $S\cong\mathbb{P}^{1}\times\mathbb{P}^{1}$ be any
non-singular quadric surface in $\mathbb{P}^{3}$ such that $S\cap H$
is integral and let $\psi_{8}:\mathbb{P}^{3}\dasharrow\mathbb{P}^{1}$
be the pencil generated by $S$ and $2H$. The support of the base
scheme $\mathrm{Bs}\psi_{8}$ is a non-singular rational curve $C$,
of type $(1,1)$ when viewed in $S$ and a non-singular conic when
viewed in $\mathbb{P}^{2}$, and $(\mathbb{P}^{3},H,\psi_{8})$ is
an $H$-special del Pezzo pencil with $m=2$ to which Theorem \ref{thm:Completion}
applies to yield, for any thrifty resolution $\tau:Y\to\mathbb{P}^{3}$
of $\psi_{8}$ a completion of $\mathbb{A}^{3}=\mathbb{P}^{3}\setminus H$
into the total space of a del Pezzo fibration $\tilde{\pi}:\tilde{Y}\to\mathbb{P}^{1}$
of degree $8$ with boundary divisor 
\[
B=B_{h}\cup B_{f}=\varphi_{*}\sigma_{*}^{-1}E_{\Gamma}\cup\varphi_{*}(\sigma^{-1}(\gamma_{*}^{-1}H))
\]
and such that the restriction $\bar{\pi}:B_{h}\to\mathbb{P}^{1}$
of $\tilde{\pi}$ is a $\mathbb{P}^{1}$-fibration. \\

$\bullet$ Let $C\subset H$ be any integral curve of degree $3$,
let $S\subset\mathbb{P}^{3}$ be any non-singular cubic surface such
that $S\cap H=C$ and let $\psi_{3}:\mathbb{P}^{3}\dashrightarrow\mathbb{P}^{1}$
be the pencil generated by $S$ and $3H$. The support of the base
scheme $\mathrm{Bs}\psi_{3}$ is equal to $C$ and $(\mathbb{P}^{3},H,\psi_{3})$
is an $H$-special del Pezzo pencil with $m=3$ to which Theorem \ref{thm:Completion}
applies to yield, for any thrifty resolution $\tau:Y\to\mathbb{P}^{3}$
of $\psi_{3}$ a completion of $\mathbb{A}^{3}=\mathbb{P}^{3}\setminus H$
into the total space of a del Pezzo fibration $\tilde{\pi}:\tilde{Y}\to\mathbb{P}^{1}$
of degree $3$ with boundary divisor 
\[
B=B_{h}\cup B_{f}=\varphi_{*}\sigma_{*}^{-1}E_{\Gamma}\cup\varphi_{*}(\sigma^{-1}(\gamma_{*}^{-1}H))
\]
 such that every member of $\psi_{3}$ other than $3H$ appears as
a closed fiber of $\tilde{\pi}:\tilde{Y}\to\mathbb{P}^{1}$ and such
that every closed fiber of $\bar{\pi}:B_{h}\to\mathbb{P}^{1}$ other
than $\bar{\pi}^{-1}(\tilde{\pi}(B_{f}))$ is isomorphic to $C$. 

Noting that $C$ is a prime anti-canonical divisor on $S$ and ``reversing
the order'' of the construction yields the following proposition,
which, combined with Lemma \ref{lem:Prime-anticanonical=00003Dplane-cubic-curve},
proves Theorem A and Theorem B in the case $d=3$. 
\begin{prop}
\label{prop:dP3-completion}For every pair $(S,D)$ consisting of
a del Pezzo surface of degree $3$ and a prime anti-canonical divisor
$D$ on $S$ there exists a completion of $\mathbb{A}^{3}$ into the
total space of a del Pezzo fibration $\pi:X\to\mathbb{P}^{1}$ of
degree $3$ with boundary $B=B_{h}\cup B_{f}$ such that general fibers
of the induced morphism $\bar{\pi}:B_{h}\to\mathbb{P}^{1}$ are isomorphic
to $D$ and such that the pair $(S,D)$ equals the pair $(X_{c},B_{h,c})$
for some closed point $c\in\mathbb{P}^{1}$. 
\end{prop}

\begin{proof}
Every del Pezzo surface $S$ of degree $3$ is anti-canonically embedded
in $\mathbb{P}^{3}=\mathbb{P}(H^{0}(S,\omega_{S}^{\vee}))$ as a cubic
hypersurface. Thus, for every prime anti-canonical divisor $D$ on
$S$, there exists a unique hyperplane $H\subset\mathbb{P}^{3}$ such
that $H\cap S=D$. Applying Theorem \ref{thm:Completion} to the $H$-special
del Pezzo pencil $(\mathbb{P}^{3},H,\psi_{3})$ where $\psi_{3}$
is, as above, the pencil generated by $S$ and $3H$, yields a completion
of $\mathbb{A}^{3}=\mathbb{P}^{3}\setminus H$ with the desired properties. 
\end{proof}

\subsubsection{Examples obtained from special pencils on the weighted projective
space $\mathbb{P}(1,1,1,2)$ }

Here we consider $\mathbb{A}^{3}$ as the complement of the support
of any Weil divisor $H\cong\mathbb{P}(1,1,2)$ in the $2$-dimensional
complete linear system $|\mathcal{O}_{\mathbb{P}(1,1,1,2)}(1)|$ on
the weighted projective space $\mathbb{P}(1,1,1,2)$. The anti-canonical
model $\mathrm{Proj}(\bigoplus_{m\geq0}H^{0}(S,(\omega_{S}^{\vee})^{\otimes m}))$
of every del Pezzo surface $S$ of degree $2$ is isomorphic to a
smooth quartic hypersurface in $\mathbb{P}(1,1,1,2)$ and conversely,
a smooth quartic hypersurface $S$ of $\mathbb{P}(1,1,1,2)$ is a
del Pezzo surface of degree $2$ such that the restriction homomorphism
\[
H^{0}(\mathbb{P}(1,1,1,2),\mathcal{O}_{\mathbb{P}(1,1,1,2)}(1))\to H^{0}(S,\omega_{S}^{\vee})
\]
is bijective. Given any pair $(S,D)$ consisting of a del Pezzo surface
$S\subset\mathbb{P}(1,1,1,2)$ of degree $2$ and a prime anti-canonical
divisor $D$ on $S$, let $\psi_{2}:\mathbb{P}(1,1,1,2)\dashrightarrow\mathbb{P}^{1}$
be the pencil generated by $S$ and $4H$ where $H\in|\mathcal{O}_{\mathbb{P}(1,1,1,2)}(1)|$
is the unique hyperplane such that $S\cap H=D$. By definition, the
support of the base scheme $\mathrm{Bs}\psi_{2}$ is the curve $D$,
which is a prime anti-canonical divisor when viewed in $S$ and an
integral quartic curve when viewed in $H\cong\mathbb{P}(1,1,2)$,
and $(\mathbb{P}(1,1,1,2),H,\psi_{2})$ is an $H$-special del Pezzo
pencil with $m=4$. By applying Theorem \ref{thm:Completion} to $(\mathbb{P}(1,1,1,2),H,\psi_{2})$
we obtain the following proposition which, together with Lemma \ref{lem:Prime-anticanonical=00003Dplane-cubic-curve},
proves Theorem A and Theorem B in the case $d=2$. 
\begin{prop}
\label{prop:dP2-completion}For every pair $(S,D)$ consisting of
a del Pezzo surface of degree $2$ and a prime anti-canonical divisor
$D$ on $S$ there exists a completion of $\mathbb{A}^{3}$ into the
total space of a del Pezzo fibration $\pi:X\to\mathbb{P}^{1}$ of
degree $2$ with boundary $B=B_{h}\cup B_{f}$ such that general fibers
of the induced morphism $\bar{\pi}:B_{h}\to\mathbb{P}^{1}$ are isomorphic
to $D$ and such that the pair $(S,D)$ equals the pair $(X_{c},B_{h,c})$
for some closed point $c\in\mathbb{P}^{1}$. 
\end{prop}

\subsubsection{Examples obtained from special pencils on the weighted projective
space $\mathbb{P}(1,1,2,3)$}

The construction is similar to that in the previous case. We consider
$\mathbb{A}^{3}$ as the complement of the support of any Weil divisor
$H\cong\mathbb{P}(1,2,3)$ in the $1$-dimensional complete linear
system $|\mathcal{O}_{\mathbb{P}(1,1,2,3)}(1)|$ on the weighted projective
space $\mathbb{P}(1,1,2,3)$. The anti-canonical model $\mathrm{Proj}(\bigoplus_{m\geq0}H^{0}(S,(\omega_{S}^{\vee})^{\otimes m}))$
of a del Pezzo surface $S$ of degree $1$ is isomorphic to a smooth
sextic hypersurface in $\mathbb{P}(1,1,2,3)$ and conversely, every
smooth sextic hypersurface $S$ of $\mathbb{P}(1,1,2,3)$ is a del
Pezzo surface of degree $1$ such that the restriction homomorphism
\[
H^{0}(\mathbb{P}(1,1,2,3),\mathcal{O}_{\mathbb{P}(1,1,2,3)}(1))\to H^{0}(S,\omega_{S}^{\vee})
\]
is bijective. Given any pair $(S,D)$ consisting of a del Pezzo surface
$S\subset\mathbb{P}(1,1,2,3)$ of degree $1$ and a prime anti-canonical
divisor $D$ on $S$, let $\psi_{1}:\mathbb{P}(1,1,2,3)\dashrightarrow\mathbb{P}^{1}$
be the pencil generated by $S$ and $6H$ where $H\in|\mathcal{O}_{\mathbb{P}(1,1,2,3)}(1)|$
is the unique hyperplane such that $S\cap H=D$. By definition, the
support of the base scheme $\mathrm{Bs}\psi_{1}$ is the curve $D$,
which is a prime anti-canonical divisor when viewed in $S$ and an
integral sextic curve when viewed in $H\cong\mathbb{P}(1,2,3)$, and
$(\mathbb{P}(1,1,2,3),H,\psi_{1})$ is an $H$-special del Pezzo pencil
with $m=6$. As in the previous paragraph, applying Theorem \ref{thm:Completion}
to $(\mathbb{P}(1,1,2,3),H,\psi_{1})$ gives the following proposition,
which, when combined with Lemma \ref{lem:Prime-anticanonical=00003Dplane-cubic-curve},
proves Theorem A and Theorem B in the case $d=1$. 
\begin{prop}
\label{prop:dP1-completion}For every pair $(S,D)$ consisting of
a del Pezzo surface of degree $1$ and a prime anti-canonical divisor
$D$ on $S$ there exists a completion of $\mathbb{A}^{3}$ into the
total space of a del Pezzo fibration $\pi:X\to\mathbb{P}^{1}$ of
degree $1$ with boundary $B=B_{h}\cup B_{f}$ such that general fibers
of the induced morphism $\bar{\pi}:B_{h}\to\mathbb{P}^{1}$ are isomorphic
to $D$ and such that the pair $(S,D)$ equals the pair $(X_{c},B_{h,c})$
for some closed point $c\in\mathbb{P}^{1}$. 
\end{prop}

\subsection{del Pezzo fibration completions from special pencils on the smooth
quadric threefold in $\mathbb{P}^{4}$}

Here we consider $\mathbb{A}^{3}$ as the complement of a singular
hyperplane section $H\cong\mathbb{P}(1,1,2)$ in a smooth quadric
threefold $Q\subset\mathbb{P}^{4}$. Recall that writing $\mathcal{O}_{Q}(1)=\mathcal{O}_{\mathbb{P}^{4}}(1)|_{Q}$,
these singular hyperplane sections are parametrized by the closed
points of the quadric hypersurface $Q^{\vee}\subset\mathbb{P}(H^{0}(Q,\mathcal{O}_{Q}(1))^{\vee})$
dual to $Q$. \\

$\bullet$ Let $H'\cong\mathbb{P}^{1}\times\mathbb{P}^{1}$ be a non-singular
hyperplane section of $Q$ such that $H\cap H'$ is integral, i.e.
a non-singular rational curve of type $(1,1)$ when viewed in $H'$,
and let $\psi_{8}:Q\dasharrow\mathbb{P}^{1}$ be the pencil generated
by $H$ and $H'$. Then $(Q,H,\psi_{8})$ is an $H$-special del Pezzo
pencil with $m=1$ which gives rise to a completion of $\mathbb{A}^{3}=Q\setminus H$
into the total space of a del Pezzo fibration $\mathrm{q}:\Gamma\to\mathbb{P}^{1}$
of degree $8$ with boundary divisor $B=B_{h}\cup B_{f}=E_{\Gamma}\cup\gamma_{*}^{-1}H$,
where the graph resolution $\gamma:\Gamma\to Q$ is the blow-up of
$Q$ along $H\cap H'$ with exceptional divisor $E_{\Gamma}\cong\mathbb{P}(\mathcal{O}_{\mathbb{P}^{1}}(-2)\oplus\mathcal{O}_{\mathbb{P}^{1}}(-2))$
and where the restriction $\bar{\mathrm{q}}:B_{h}\to\mathbb{P}^{1}$
is a $\mathbb{P}^{1}$-bundle. \\

$\bullet$ Let $S\subset Q$ be a non-singular quadric section of
$Q$ such that $H\cap S$ is integral and let $\psi_{4}:Q\dashrightarrow\mathbb{P}^{1}$
be the pencil generated by $S$ and $2H$. Then $S$ is a del Pezzo
surface of degree $4$, the base locus $\mathrm{Bs}\psi_{4}$ is supported
by $D=H\cap S$, which is an integral anti-canonical divisor on $S$
and $H$, and $(Q,H,\psi_{4})$ is an $H$-special del Pezzo pencil
with $m=2$ to which Theorem \ref{thm:Completion} applies to yield,
for any thrifty resolution $\tau:Y\to Q$ of $\psi_{4}$, a completion
of $\mathbb{A}^{3}=Q\setminus H$ into the total space of a del Pezzo
fibration $\tilde{\pi}:\tilde{Y}\to\mathbb{P}^{1}$ of degree $4$
with boundary divisor 
\[
B=B_{h}\cup B_{f}=\varphi_{*}\sigma_{*}^{-1}E_{\Gamma}\cup\varphi_{*}(\sigma^{-1}(\gamma_{*}^{-1}H))
\]
such that every member of $\psi_{4}$ other than $2H$ appears as
a closed fiber of $\tilde{\pi}:\tilde{Y}\to\mathbb{P}^{1}$ and such
that every closed fiber of $\bar{\pi}:B_{h}\to\mathbb{P}^{1}$ other
than $\bar{\pi}^{-1}(\tilde{\pi}(B_{f}))$ is isomorphic to $D$.
Building on this construction, we obtain the following proposition
which proves Theorem A and Theorem B in the case $d=4$. 
\begin{prop}
\label{prop:dP4-completion}With the notation above, the following
hold:

a) For every del Pezzo surface $S$ of degree $4$, there exists a
completion of $\mathbb{A}^{3}$ into the total space of a del Pezzo
fibration $\tilde{\pi}:\tilde{Y}\to\mathbb{P}^{1}$ of degree $4$
such that $S$ appears as a closed fiber of $\tilde{\pi}$.

b) For every integral plane cubic curve $C$, there exist a pair $(S,D)$
consisting of a smooth del Pezzo surface of degree $4$, an anti-canonical
divisor $D\cong C$ and a completion of $\mathbb{A}^{3}$ into the
total space of a del Pezzo fibration $\tilde{\pi}:\tilde{Y}\to\mathbb{P}^{1}$
of degree $4$ with boundary $B=B_{h}\cup B_{f}$ such that the pair
$(S,D)$ equals the pair $(\tilde{Y}_{c},B_{h,c})$ for some closed
point $c\in\mathbb{P}^{1}$. 
\end{prop}

\begin{proof}
Every del Pezzo surface $S$ of degree $4$ is anti-canonically embedded
in $\mathbb{P}^{4}=\mathbb{P}(H^{0}(S,\omega_{S}^{\vee}))$ as the
intersection of two quadric hypersurfaces. Since a general member
of the pencil of quadrics containing $S$ is non-singular, we can
assume without loss of generality that $S=Q\cap Q'$ where $Q$ and
$Q'$ are smooth quadric threefolds. Since the restriction homomorphism
$H^{0}(Q,\mathcal{O}_{Q}(1))\to H^{0}(S,\omega_{S}^{\vee})$ is an
isomorphism, for every effective anti-canonical divisor $D$ on $S$,
there exists a unique hyperplane section $H_{D}$ of $Q$ such that
$S\cap H_{D}=D$. To prove Assertion a), it thus suffices to ascertain
the existence of prime anti-canonical divisors $D$ on $S$ such that
$H_{D}$ is a singular hyperplane section of $Q$ to which the construction
described above can be applied. For this, we first observe that non-prime
anti-canonical divisors $D$ on $S$ are parametrized by the closed
points of a sub-scheme $Z$ of $\mathbb{P}(H^{0}(S,\omega_{S}^{\vee})^{\vee})=\mathbb{P}(H^{0}(Q,\mathcal{O}_{Q}(1))^{\vee})$
of codimension $2$. Namely, the image $\alpha_{*}D$ of any effective
anti-canonical divisor $D$ on $S$ by a birational morphism $\alpha:S\to\mathbb{P}^{2}$
expressing $S$ as the blow-up of $\mathbb{P}^{2}$ at five distinct
points $p_{i}$ is a cubic curve passing through all the $p_{i}$.
If $D$ is not prime but $\alpha_{*}D$ is prime, then $\alpha_{*}D$
is an integral rational cubic curve singular at one of the $p_{i}$,
and a straightforward dimension count shows that such curves $\alpha_{*}D$
 are parametrized by a sub-scheme of dimension $2$ of $\mathbb{P}(H^{0}(\mathbb{P}^{2},\mathcal{O}_{\mathbb{P}^{2}}(3))^{\vee})$.
Otherwise, if $\alpha_{*}D$ is not prime then it consists of the
union of a line $L$ and a reduced conic $C$, possibly reducible,
with the following possible configurations: 1) $L$ does not pass
through any of the $p_{i}$'s and $C$ is the unique conic passing
through all the $p_{i}$'s; 2) $L$ passes through a unique point
$p_{i}$ and $C$ passes through all the remaining points; 3) $L$
passes through $2$ distinct points $p_{i}$ and $p_{j}$ and $C$
passes through all the remaining points. The curves $\alpha_{*}D$
corresponding to the first configuration are parametrized by an open
sub-scheme of $\mathbb{P}(H^{0}(\mathbb{P}^{2},\mathcal{O}_{\mathbb{P}^{2}}(1))^{\vee})$,
those corresponding to the second configuration are parametrized by
a finite union of open sub-schemes $Z_{i}$ of closed sub-schemes
of $\mathbb{P}(H^{0}(\mathbb{P}^{2},\mathcal{O}_{\mathbb{P}^{2}}(1))^{\vee})\times\mathbb{P}(H^{0}(\mathbb{P}^{2},\mathcal{O}_{\mathbb{P}^{2}}(2))^{\vee})$
isomorphic to $\mathbb{P}^{1}\times\mathbb{P}^{1}$, and those corresponding
to the third configuration are parametrized by a finite union of closed
sub-schemes $Z'_{ij}\cong\mathbb{P}^{2}$ of $\mathbb{P}(H^{0}(\mathbb{P}^{2},\mathcal{O}_{\mathbb{P}^{2}}(2))^{\vee})$.
On the other hand, we noticed earlier that singular hyperplane sections
of $Q$ are parametrized by the points of the quadric hypersurface
$Q^{\vee}\subset\mathbb{P}(H^{0}(Q,\mathcal{O}_{Q}(1))^{\vee})$.
Therefore, any closed point of the non-empty open subset $Q^{\vee}\setminus(Q^{\vee}\cap Z)$
corresponds to a prime anti-canonical divisor $D$ on $S$ for which
the unique hyperplane section $H_{D}$ of $Q$ such that $S\cap H_{D}=D$
is singular, and Assertion a) follows.

To prove Assertion b) we fix a singular hyperplane section $H\cong\mathbb{P}(1,1,2)$
in a smooth quadric threefold $Q\subset\mathbb{P}^{4}$. Every isomorphism
type of integral plane cubic curve $C$ can be realized by an integral
quartic curve $D$ contained in the regular locus of $H$, the curve
$D$ being then an anti-canonical divisor on $H$. From the exact
sequence of sheaves $0\to\mathcal{I}_{H}(2)\to\mathcal{O}_{Q}(2)\to\mathcal{O}_{Q}(2)|_{H}\cong\omega_{H}^{\vee}\to0$
and the vanishing of $H^{1}(Q,\mathcal{I}_{H}(2))\cong H^{1}(Q,\mathcal{O}_{Q}(1))$,
we infer that the restriction homomorphism $H^{0}(Q,\mathcal{O}_{Q}(2))\to H^{0}(H,\omega_{H}^{\vee})$
is surjective with kernel isomorphic to $H^{0}(Q,\mathcal{O}_{Q}(1))$.
Since $D$ is integral, there exists a section $s\in H^{0}(Q,\mathcal{O}_{Q}(2))$
with integral zero locus $Z(s)$ in $Q$ such that $Z(s)\cap H=D$
and, letting $h\in H^{0}(Q,\mathcal{O}_{Q}(1))$ be a section with
zero locus $H$, it follows that every other section $s'\in H^{0}(Q,\mathcal{O}_{Q}(2))$
such that $Z(s')\cap H=D$ is, up to a nonzero scalar multiple, of
the form $s'=s+hf$ for some global section $f\in H^{0}(Q,\mathcal{O}_{Q}(1))$.
Since $D$ is a Cartier divisor on $Z(s')$, $Z(s')$ is regular along
$D_{\mathrm{reg}}$. Moreover, if $D$ is singular then the fact that
the support of $D$ is contained in the regular locus of $H$ guarantees
that the open subset of sections $f\in H^{0}(Q,\mathcal{O}_{Q}(1))$
such that $Z(s')$ is regular at the unique singular point $x$ of
$D$ is nonempty. Indeed, letting $\mathfrak{m}_{x}\subset\mathcal{O}_{Q,x}$
be the maximal ideal, we have $h\in H^{0}(Q,\mathcal{O}_{Q}(1)\otimes\mathfrak{m}_{x})\setminus H^{0}(Q,\mathcal{O}_{Q}(1)\otimes\mathfrak{m}_{x}^{2})$
and hence, if $s\in H^{0}(Q,\mathcal{O}_{Q}(2)\otimes\mathfrak{m}_{x}^{2})$
then $s+hf\in H^{0}(Q,\mathcal{O}_{Q}(2)\otimes\mathfrak{m}_{x})\setminus H^{0}(Q,\mathcal{O}_{Q}(2)\otimes\mathfrak{m}_{x}^{2})$
for every $f\in H^{0}(Q,\mathcal{O}_{Q}(1))$ not vanishing at $x$.
Since on the other hand, by Bertini's theorem, the zero locus of a
general section $s'=s+hf$ is regular outside the support of $D$,
we conclude that for a general choice of section $f$, the zero locus
$Z(s')$ in $Q$ of the corresponding section $s'$ is regular, hence
is a del Pezzo surface $S$ of degree $4$ such that $S\cap H=D$.
Assertion b) then follows from Theorem \ref{thm:Completion} applied
to the $H$-special del Pezzo pencil $(Q,H,\psi)$ where $\psi$ is
the pencil generated by $S$ and $2H$. 
\end{proof}
The proof of Proposition \ref{prop:dP4-completion} shows that for
a pair $(S,D)$ consisting of a del Pezzo surface $S$ of degree $4$
and a \emph{general} prime anti-canonical divisor $D$ on $S$, there
exists a completion of $\mathbb{A}^{3}$ into the total space of a
 del Pezzo fibration $\pi:X\to\mathbb{P}^{1}$ of degree $4$ with
boundary $B=B_{h}\cup B_{f}$ such that general fibers of the induced
morphism $\bar{\pi}:B_{h}\to\mathbb{P}^{1}$ are isomorphic to $D$
and such that the pair $(S,D)$ equals the pair $(X_{c},B_{h,c})$
for some closed point $c\in\mathbb{P}^{1}$. So in contrast with the
cases $d=1,2,3$ considered above, the following question remains
open:
\begin{question}
Let $(S,D)$ be a pair consisting of a del Pezzo surface of degree
$4$ and a prime anti-canonical divisor on $S$. Does there exist
a completion of $\mathbb{A}^{3}$ into the total space of a del Pezzo
fibration $\pi:X\to\mathbb{P}^{1}$ of degree $4$ with boundary $B=B_{h}\cup B_{f}$
such that general fibers of the induced morphism $\bar{\pi}:B_{h}\to\mathbb{P}^{1}$
are isomorphic to $D$ and such that the pair $(S,D)$ equals the
pair $(X_{c},B_{h,c})$ for some closed point $c\in\mathbb{P}^{1}$
? 
\end{question}

\begin{rem}
Other examples of completions of $\mathbb{A}^{3}$ into total spaces
of del Pezzo fibrations $\pi:X\to\mathbb{P}^{1}$ of degree $4$ were
already constructed in \cite[Section 2.2.]{DK18} as outputs of non-trivial
explicit relative MMP's ran from certain del Pezzo pencils of degree
$3$ over $\mathbb{P}^{1}$ with generic fiber of Picard rank two.
Let us also notice that in the construction given in the proof of
Proposition \ref{prop:dP4-completion}, we only considered prime anti-canonical
divisors on a singular hyperplane section $H\cong\mathbb{P}(1,1,2)$
whose supports are contained in the regular locus of $H$. But $H$
also carries another family of prime anti-canonical divisors $D$
whose supports are singular rational curves, either nodal or cuspidal,
passing through the singular point $p$ of $H$, namely, the images
by the minimal desingularization morphism $\beta:\mathbb{F}_{2}=\mathbb{P}(\mathcal{O}_{\mathbb{P}^{1}}\oplus\mathcal{O}_{\mathbb{P}^{1}}(-2))\to H$
of integral members of the linear system $|s_{2}+4f_{2}|$, where
$s_{2}$ and $f_{2}$ denote respectively the negative section and
a fiber of the $\mathbb{P}^{1}$-bundle structure $\mathbb{F}_{2}\to\mathbb{P}^{1}$.
For every such prime anti-canonical divisor $D$ passing through $p$,
one can argue the existence of a smooth quadric section $S$ of $Q$
such that $S\cap H=D$ and hence, that $(Q,H,\psi)$, where $\psi$
is the pencil generated by $S$ and $2H$, is an $H$-special del
Pezzo pencil which gives rise by Theorem \ref{thm:Completion} to
a completion of $\mathbb{A}^{3}$ into the total space of a del Pezzo
fibration $\tilde{\pi}:\tilde{Y}\to\mathbb{P}^{1}$ with boundary
$B=B_{h}\cup B_{f}$ such that the pair $(S,D)$ equals the pair $(X_{c},B_{h,c})$
for some closed point $c\in\mathbb{P}^{1}$. 
\end{rem}

\subsection{del Pezzo fibration completions from special pencils on the quintic
del Pezzo threefold}

Recall that a quintic del Pezzo threefold is a smooth Fano threefold
$V_{5}$ of Picard rank one, index $2$ and degree $-\frac{1}{8}(K_{V_{5}}^{3})=5$.
By \cite{Fuj81}, a threefold with these invariants is unique up to
isomorphism and isomorphic to any non-singular section of the Grassmannian
$\mathrm{Gr}(2,5)\subset\mathbb{P}^{9}$ by a linear subspace of codimension
$3$. Viewing $V_{5}$ as a closed sub-variety of $\mathbb{P}^{6}$,
every non-singular hyperplane section of $V_{5}$ is a del Pezzo surface
of degree $5$. On the other hand, by \cite[Theorem 1]{Fur86}, there
exists up to automorphisms of $V_{5}$ a unique normal hyperplane
sections $H$ of $V_{5}$ with a unique singular point $p$ of type
$A_{4}$ such that $V_{5}\setminus H$ is isomorphic to $\mathbb{A}^{3}$.
The following principle of construction was already described in \cite[Example 6.3]{DKP22}:
\\

$\bullet$ Let $S\subset V_{5}$ be a non-singular hyperplane section
of $V_{5}$ such that $D=H\cap S$ is integral and let $\psi_{5}:V_{5}\dashrightarrow\mathbb{P}^{1}$
be the pencil generated by $S$ and $H$. Then $S$ is a del Pezzo
surface of degree $5$, the base locus $\mathrm{Bs}\psi_{5}$ equals
$D$, which is an integral anti-canonical divisor on $S$ and $H$,
and $(V_{5},H,\psi_{5})$ is an $H$-special del Pezzo pencil with
$m=1$ which gives rise to a completion of $\mathbb{A}^{3}=V_{5}\setminus H$
into the total space of a del Pezzo fibration $\mathrm{q}:\Gamma\to\mathbb{P}^{1}$
of degree $5$ with boundary divisor $B=B_{h}\cup B_{f}=E_{\Gamma}\cup\gamma_{*}^{-1}H$,
where the graph resolution $\gamma:\Gamma\to V_{5}$ is the blow-up
of $V_{5}$ along $D$. Noting that a del Pezzo surface of degree
$5$ is unique up to isomorphism, say isomorphic to $\mathbb{P}^{2}$
blown-up at the four points $[0:0:1]$, $[0:1:0]$, $[1:0:0]$ and
$[1:1:1]$, the following proposition proves Theorem A and Theorem
B in the case $d=5$. 
\begin{prop}
\label{prop:dP5-completion}For every integral plane cubic curve $C$,
there exist a pair $(S,D)$ consisting of a del Pezzo surface of degree
$5$ and an anti-canonical divisor $D\cong C$ and a completion of
$\mathbb{A}^{3}$ into the total space of a del Pezzo fibration $\pi:X\to\mathbb{P}^{1}$
of degree $5$ with boundary $B=B_{h}\cup B_{f}$ such that the pair
$(S,D)$ equals the pair $(X_{c},B_{h,c})$ for some closed point
$c\in\mathbb{P}^{1}$. 
\end{prop}

\begin{proof}
We first observe that for every integral plane cubic curve $C$, there
exists a prime anti-canonical divisor $D$ on $H$ isomorphic to $C$
and whose support is contained in the regular locus of $H$. Indeed,
by \cite[Proposition 15]{Fur86}, $H$ contains a unique line $L$,
passing through the singular point $p$ of $H$, and the proper transform
$L'$ in the minimal resolution of singularities $\beta:\tilde{H}\to H$
is a $(-1)$-curve meeting the chain $E_{1}+E_{2}+E_{3}+E_{4}$ of
$(-2)$-curves forming the exceptional locus of $\beta$ normally
at a point of $E_{3}$. Let $\alpha:\tilde{H}\to\mathbb{P}^{2}$ be
the birational morphism which contracts $L'\cup E_{3}\cup E_{2}\cup E_{1}$
to a point $q$ and maps $E_{4}$ onto a line $\ell$. Every isomorphism
type of integral cubic curve, either an elliptic curve or an integral
rational curve with a double point can be realized by a curve $C$
which intersects $\ell$ with multiplicity $3$ at $q$ and whose
proper transform $\alpha_{*}^{-1}C$ in $\tilde{H}$ meets $L'\cup E_{3}\cup E_{2}\cup E_{1}\cup E_{4}$
normally at a unique point of $L'$. The image $D=\beta_{*}(\alpha_{*}^{-1}C)$
is then a prime anti-canonical divisor on $H$ contained in the regular
locus of $H$, isomorphic to $C$ and meeting $L$ normally at a single
point. Letting $\mathcal{O}_{V_{5}}(1)=\mathcal{O}_{\mathbb{P}^{6}}(1)|_{V_{5}}$,
the exact sequence of sheaves $0\to\mathcal{I}_{H}(1)\cong\mathcal{O}_{V_{5}}\to\mathcal{O}_{V_{5}}(1)\to\mathcal{O}_{V_{5}}(1)|_{H}\cong\omega_{H}^{\vee}\to0$
together with the vanishing of $H^{1}(V_{5},\mathcal{O}_{V_{5}})$
imply that for every prime anti-canonical divisor $D$ on $H$ as
above, there exists up to isomorphism a unique pencil of hyperplane
sections $\psi:V_{5}\dashrightarrow\mathbb{P}^{1}$ with base locus
$\mathrm{Bs}\psi$ equal to $D$ and containing $H$ among its members.
Since $D$ is contained in the regular locus of $H$, a general member
of $\psi$ is regular, hence is a del Pezzo surface $S$ of degree
$5$. Thus, $(V_{5},H,\psi)$ is an $H$-special del Pezzo pencil
and the assertion follows from Theorem \ref{thm:Completion}. 
\end{proof}
\begin{rem}
In the proof of Proposition \ref{prop:dP5-completion}, we only considered
prime anti-canonical divisors on a normal hyperplane section $H$
of $V_{5}$ with a unique singular point $p$ of type $A_{4}$ whose
supports are contained in the regular locus of $H$. But $H$ also
carries another family of prime anti-canonical divisors $D$ whose
supports are nodal (resp. cuspidal) rational curves passing through
$p$: keeping the notation of the proof of Proposition \ref{prop:dP5-completion},
these are the proper transforms $\beta_{*}(\alpha_{*}^{-1}C)$ of
irreducible conics in $\mathbb{P}^{2}$ intersecting $\ell$ normally
at $q$ (resp. with multiplicity $2$ at $q$). For every such prime
anti-canonical divisor $D$ passing through $p$, a general member
$S$ of the unique pencil of hyperplane sections $\psi:V_{5}\dashrightarrow\mathbb{P}^{1}$
with base locus $\mathrm{Bs\psi}=D$ is regular away from $p$ by
Bertini\textquoteright s theorem and regular at $p$ since it intersects
the unique line $L$ in $H$ through $p$ with multiplicity $1$.
Thus, $S$ is a del Pezzo surface of degree $5$ and $(V_{5},H,\psi)$
is an $H$-special del Pezzo pencil whose graph resolution $\gamma:\Gamma\to V_{5}$
is the blow-up of $V_{5}$ along $D$ and for which $\mathrm{q}:\Gamma\to\mathbb{P}^{1}$
is a del Pezzo fibration of degree $5$ with $\Gamma\setminus(E_{\Gamma}\cup\gamma_{*}^{-1}H)\cong\mathbb{A}^{3}$. 
\end{rem}

\section{Two constructions of completions into total spaces of del Pezzo fibrations
of degree $6$}

In the next subsections, we present new constructions of completions
of $\mathbb{A}^{3}$ into total spaces of del Pezzo fibrations of
degree $6$ over $\mathbb{P}^{1}$ from which we are able to infer
Theorem A and Theorem B in the remaining case $d=6$. 
\begin{notation}
Throughout this section, given a closed sub-scheme $Z$ of a normal
variety $Y$ and a birational model $\tilde{Y}$ of $Y$, we denote
by $Z_{\tilde{Y}}$ the proper transform of $Z$ in $\widetilde{Y}$.
\end{notation}

\subsection{\label{sec:deg6nonnormal}del Pezzo completions $(X,B_{h}\cup B_{f})$
of degree $6$ with induced rational fibrations $\bar{\pi}:B_{h}\to\mathbb{P}^{1}$ }

In this subsection, we construct examples of completions of $\mathbb{A}^{3}$
into total spaces of del Pezzo fibrations $\pi:X\to\mathbb{P}^{1}$
of degree $6$ with boundary $B=B_{h}\cup B_{f}$ for which a general
fiber of the induced morphism $\bar{\pi}:B_{h}\to\mathbb{P}^{1}$
is isomorphic to a nodal or cuspidal plane cubic curve. In contrast
with the constructions made in the previous sections which involve
birational maps between Fano threefolds of divisor class rank one
and del Pezzo fibrations which preserve a given open subset isomorphic
to $\mathbb{A}^{3}$, the construction below builds instead on the
existence of Sarkisov links between two del Pezzo fibrations of different
degrees which induce nontrivial affine blow-ups between pairs of affine
open subsets isomorphic to $\mathbb{A}^{3}$. Recall \cite{KZ99},
\cite[\href{https://stacks.math.columbia.edu/tag/052P}{Tag 052P}]{StackPro}
that given a triple $(C,S,V)$ consisting of an affine variety $V=\mathrm{Spec}(A)$,
a closed sub-scheme $C$ with defining ideal $I\subset A$ and a principal
Cartier divisor $S=\mathrm{div}(f)$ where $f\in I$, the \emph{affine
blow-up $\mathrm{Bl}_{C}^{S}V$ of $V$ with center at $C$ and divisor
$S$} is the affine variety defined as the spectrum of the ring $A[I/f]:=(\bigoplus_{m\geq0}I^{m}t^{m})/(1-ft)$,
and that in the case where the closed immersions $C\subset S\subset V$
are regular, $\mathrm{Bl}_{C}^{S}V$ identifies with the complement
in the blow-up $\tilde{V}\to V$ of $V$ with center at $C$ of the
proper transform of the divisor $S$.

\subsubsection{\label{sec:(2,3)-construction}del Pezzo fibrations of degree $6$
arising from conic and twisted cubic curves in the smooth quadric
threefold}

Let $Q\subset\mathbb{P}^{4}$ be a smooth quadric threefold, let $C_{i}\subset Q$,
$i=2,3$, be a pair of smooth rational curves of degree $i$ such
that $C_{2}\cap C_{3}=\emptyset$ and let $H_{2}\subset Q$ and $H_{3}\subset Q$
be respectively a hyperplane section of $Q$ containing $C_{2}$ and
the unique hyperplane section of $Q$ containing $C_{3}$. Let $h_{i}\colon Q_{i}\to Q$,
$i=2,3$ be the blow-up along $C_{i}$ with exceptional divisor $E_{i}$,
so that $Q_{2}$ (resp. $Q_{3}$) is the Fano threefold No.29 (resp.
No.26) in \cite[Table 2]{M-M81}. The linear system $|h_{2}^{*}H_{2}-E_{2}|$
defines a quadric fibration $q_{2}\colon Q_{2}\to\mathbb{P}^{1}$
whereas the linear system $|2h_{3}^{*}H_{3}-E_{3}|$ defines a birational
morphism $q_{3}\colon Q_{3}\to V_{5}$ to the smooth quintic del Pezzo
threefold $V_{5}$ with exceptional locus $(H_{3})_{Q_{3}}$. Put
$V=Q\setminus H_{2}$, $S=H_{3}\cap V$ and $C=C_{3}\cap V$ and let
$\mathrm{Bl}_{C}^{S}V$ be the affine blow-up of $V$ with center
at $C$ and divisor $S$. \\

Now let $Y$ be the blow-up of $Q$ along $C_{2}\sqcup C_{3}$ and
let $f_{i}\colon Y\to Q_{i}$ be the induced morphisms. The linear
equivalences 
\[
-K_{Y}\sim3(h_{2}\circ f_{2})^{*}(H_{3})-(E_{2})_{Y}-(E_{3})_{Y}\sim f_{2}^{*}(h_{2}^{*}H_{2}-E_{2})+f_{3}^{*}(2h_{3}^{*}H_{3}-E_{3})
\]
and the equality $(-K_{Y})^{3}=(-K_{\mathbb{Q}^{3}})^{3}-2((-K_{Q}\cdot C_{2}+C_{3})+2)=20>0$
imply that $-K_{Y}$ is nef and big, i.e., $Y$ is a smooth weak Fano
threefold of Picard rank three, and that each $(-K_{Y})$-trivial
curve is $(q_{i}\circ f_{i})$-exceptional for $i=2,3$. By \cite[Proposition 3.5(2)]{Fuk17}
and \cite[Proposition 2.1(2)]{Fuk18}, there exists a Sarkisov link 

\begin{equation}\label{eq:link}
\xymatrix{ & Y \ar[dl]_{f_2} \ar@{.>}[r]^{\chi}  & Y^+ \ar[dr]^{g} \\ Q_2 \ar[d]_{q_2} & & & X \ar[d]^{\pi} \\ \mathbb{P}^1 \ar@{=}[rrr] & & & \mathbb{P}^1 }
\end{equation} where $\chi$ is the flop of the $(-K_{Y})$-trivial curves, $\pi:X\to\mathbb{P}^{1}$
is a del Pezzo fibration of degree $6$ and $g$ is the blow-up of
$X$ along a section of $\pi$ with exceptional divisor $(H_{3})_{Y^{+}}$. 
\begin{lem}
\label{prop:non-normalA3} The open subset $X\setminus((E_{2})_{X}\cup(H_{2})_{X})$
is isomorphic to the affine blow-up $\mathrm{Bl}_{C}^{S}V$ of $V$.
\end{lem}

\begin{proof}
By construction, $Q_{2}\setminus(E_{2}\cup(H_{2})_{Q_{2}})\cong V$
and $U=Y\setminus((E_{2})_{Y}\cup(H_{2}\cup H_{3})_{Y})$ is isomorphic
to the affine blow-up $\mathrm{Bl}_{C}^{S}V$. Since the $(q_{3}\circ f_{3})$-exceptional
divisor is $(E_{2})_{Y}\cup(H_{3})_{Y}$ and the flopping curves of
$\chi$ are $(q_{i}\circ f_{i})$-exceptional, it follows that $\chi$
restricts to an isomorphism between $U$ and its image $\chi(U)$
and that the complement $Y^{+}\setminus\chi(U)$ is the union of $(E_{2})_{Y^{+}}\cup(H_{2}\cup H_{3})_{Y^{+}}$
and of all the flopped curves. Since $\chi(U)$ is affine, its complement
in $Y^{+}$ is of pure codimension one, which implies that $U\cong\chi(U)=Y^{+}\setminus((E_{2})_{Y^{+}}\cup(H_{2}\cup H_{3})_{Y^{+}})\cong X\setminus((E_{2})_{X}\cup(H_{2})_{X})$. 
\end{proof}
With the notation above, we observe that if the triple $(C\subset S\subset V)$
is isomorphic to a triple $(\mathbb{A}^{1}\subset\mathbb{A}^{2}\subset\mathbb{A}^{3})$
where $\mathbb{A}^{2}\subset\mathbb{A}^{3}$ is a closed immersion
as a linear subspace (recall that by the Abhyankar-Moh theorem \cite{AM75}
every closed immersion $\mathbb{A}^{1}\subset\mathbb{A}^{2}$ is equivalent
to a linear one), then $\mathrm{Bl}_{C}^{S}V$ is isomorphic to $\mathbb{A}^{3}$.
To conclude the construction, it thus remains to find suitable configurations
of curves $C_{i}$ and hyperplane sections $H_{i}$, $i=2,3$, for
which the associated triple $(C\subset S\subset V)$ is isomorphic
to a triple $(\mathbb{A}^{1}\subset\mathbb{A}^{2}\subset\mathbb{A}^{3})$
with the above property:

\vspace{0.2cm}

(1) If the hyperplane section $H_{3}$ containing $C_{3}$ is a quadric
cone, then there exists a unique line $\ell$ of the ruling of $H_{3}$
which intersects $C_{3}$ only at the singular point $p$ of $H_{3}$.
By choosing for the hyperplane section $H_{2}$ of $Q_{2}$ a quadric
cone other than $H_{3}$ with vertex on $\ell\setminus\{p\}$ we obtain
a triple $(C,S,V)=(C_{3}\setminus\{p\},H_{3}\setminus\ell,Q\setminus H_{2})$
isomorphic to $(\mathbb{A}^{1}\subset\mathbb{A}^{2}\subset\mathbb{A}^{3})$
where $\mathbb{A}^{2}\subset\mathbb{A}^{3}$ is a closed immersion
as a linear subspace. Any smooth conic $C_{2}\subset H_{2}$ disjoint
from $C_{3}$ intersects $H_{3}$ with multiplicity two at a single
point and the corresponding del Pezzo fibration $\pi:X\to\mathbb{P}^{1}$
of degree $6$ is a completion of $\mathbb{A}^{3}\cong\mathrm{Bl}_{C}^{S}V$
with boundary divisor $B=B_{h}\cup B_{f}=(E_{2})_{X}\cup(H_{2})_{X}$
for which a general fiber of the induced morphism $\bar{\pi}:B_{h}=(E_{2})_{X}\to\mathbb{P}^{1}$
is a cuspidal rational curve.

\vspace{0.2cm}

(2) If the hyperplane section $H_{3}$ containing $C_{3}$ is nonsingular,
then again, there exists a line $\ell\subset H_{3}$ tangent to $C_{3}$
at a point $p$ and then, for $H_{2}$ equal to the quadric cone with
vertex $p$, the triple $(C,S,V)=(C_{3}\setminus\{p\},H_{3}\setminus(\ell\cup\ell'),Q\setminus H_{2})$,
where $\ell'$ is the other line in $H_{3}\cong\mathbb{P}^{1}\times\mathbb{P}^{1}$
passing through $p$, is isomorphic to $(\mathbb{A}^{1}\subset\mathbb{A}^{2}\subset\mathbb{A}^{3})$
where $\mathbb{A}^{2}\subset\mathbb{A}^{3}$ is a closed immersion
as a linear subspace. Any smooth conic $C_{2}\subset H_{2}$ disjoint
from $C_{3}$ intersects $H_{3}$ in two distinct points and the corresponding
del Pezzo fibration $\pi:X\to\mathbb{P}^{1}$ of degree $6$ is a
completion of $\mathbb{A}^{3}\cong\mathrm{Bl}_{C}^{S}V$ with boundary
divisor $B=B_{h}\cup B_{f}=(E_{2})_{X}\cup(H_{2})_{X}$ for which
a general fiber of the induced morphism $\bar{\pi}:B_{h}=(E_{2})_{X}\to\mathbb{P}^{1}$
is a nodal rational curve. 

\subsubsection{Comparison with a construction of Prokhorov starting from the quintic
del Pezzo threefold}

Keeping the same notation as in the previous subsection, we now relate
through the geometry of the Sarkisov link $V_{5}\stackrel{q_{3}}{\leftarrow}Q_{3}\stackrel{h_{3}}{\rightarrow}Q$
the constructions (1) and (2) above to another example of completions
of $\mathbb{A}^{3}$ into total spaces of del Pezzo fibrations of
degree $6$ constructed by Prokhorov \cite[Theorem 6.6]{Pro16}. Letting
$H$ be a general hyperplane section of $V_{5}$, we have $q_{3}^{*}H\sim2h_{3}^{*}H_{3}-E_{3}$
and $q_{3}$ contracts $(H_{3})_{Q_{3}}$ to a line $\Upsilon\subset V_{5}$.
It is readily checked that the proper transform $(C_{2})_{V_{5}}$
of a conic $C_{2}\subset H_{2}$ disjoint from $C_{3}$ is a smooth
rational quartic curve which is bi-secant to $\Upsilon$ and that
the proper transform $(H_{2})_{V_{5}}$ of $H_{2}$ is a hyperplane
section of $V_{5}$ containing $(C_{2})_{V_{5}}\cup\Upsilon$. Conversely,
the proper transform $(C_{4})_{Q}$ of every smooth rational quartic
curve $C_{4}\subset V_{5}$ which is bi-secant to $\Upsilon$ is a
conic disjoint from $C_{3}$. By \cite[Proposition 6.3]{Pro16} the
total space $Q'$ of the blow-up $h'\colon Q'\to V_{5}$ of $V_{5}$
along such a quartic rational curve $C_{4}$ is a weak Fano threefold
and $\Upsilon_{Q'}$ is the unique $(-K_{Q'})$-trivial curve. Furthermore,
the target of the flop $\chi'\colon Q'\dashrightarrow X'$ of $\Upsilon_{Q'}$
is the total space of a del Pezzo fibration $\pi'\colon X'\to\mathbb{P}^{1}$
of degree $6$, where $X'$ is the threefold No.1 of \cite[Section 7.4]{JPR11}.
We thus have the following commutative diagram \[\xymatrix@!C{ &  &  &  & Y \ar[dl]^-{f_{3}}\ar[dr]_-{f_{2}}\ar@{=}[rr] &  & Y\ar[dl]^-{f_{2}}\ar@{.>}[r]^-{\chi} & Y^{+}\ar[dr]_-{g}\\ X'\ar[d]_-{\pi'} & Q'\ar[dr]^-{h'}\ar@{.>}[l]_-{\chi'} &  & Q_{3}\ar[dl]_-{q_{3}}\ar[dr]^-{h_{3}} &  & Q_{2}\ar[d]_-{q_{2}}\ar[dl]_-{h_{2}} &  &  & X\ar[d]^-{\pi}\\ \mathbb{P}^{1} &  & V_{5} &  & Q & \mathbb{P}^{1}\ar@{=}[rrr] &  &  & \mathbb{P}^{1}. } \]

Since $Y\dashrightarrow X$ and $Y\dashrightarrow X'$ both have small
indeterminacy loci and the divisorial part of their exceptional loci
equals $(H_{3})_{Y}$, the induced birational map $\Phi\colon X\dashrightarrow X'$
is an isomorphism in codimension one. Moreover, since $(H_{2})_{X}$
and $(H_{2})_{X'}$ are fibers of $\pi$ and $\pi'$ respectively,
$\Phi$ is a birational map of schemes over the base curve $\mathbb{P}^{1}$.
It then follows from \cite[(3.5) Proposition]{Cor95} that $\pi:X\to\mathbb{P}^{1}$
and $\pi':X'\to\mathbb{P}^{1}$ are isomorphic del Pezzo fibrations. 

In the case (1) of the construction given in subsection \ref{sec:(2,3)-construction}
above, Furushima \cite[p.120]{Fur00} provides explicit equations
of $H_{2},H_{3}$ and $C_{3}$ and shows that the hyperplane section
$(H_{2})_{V_{5}}$ has a Du Val $A_{4}$-singularity {[}\textit{ibid},
Lemma 11 (2){]}. On the other hand, in the case (2), $(H_{2})_{V_{5}}$
becomes a non-normal hyperplane section singular along a line and
whose normalization is the Hirzebruch surface $\mathbb{F}_{3}=\mathbb{P}_{\mathbb{P}^{1}}(\mathcal{O}_{\mathbb{P}^{1}}\oplus\mathcal{O}_{\mathbb{P}^{1}}(-3))$,
which is the same as the hyperplane section $R$ considered by Prokhorov
in the proof of \cite[Theorem 6.6]{Pro16}. 

\subsection{del Pezzo completions $(X,B_{h}\cup B_{f})$ of degree $6$ with
induced elliptic fibrations $\bar{\pi}:B_{h}\to\mathbb{P}^{1}$ }

For the construction, we will apply again the Sarkisov link \eqref{eq:link}
with center along a trisection of a quadric fibration over $\mathbb{P}^{1}$
to obtain a del Pezzo fibration $\pi:X\to\mathbb{P}^{1}$ of degree
$6$ whose total space is a completion of $\mathbb{A}^{3}$. But here,
in contrast with the construction given in subsection \ref{sec:deg6nonnormal},
the open subset of $X$ isomorphic to $\mathbb{A}^{3}$ will not arise
as an affine blow-up of an open subset isomorphic to $\mathbb{A}^{3}$
in the initial quadric fibration. 

\subsubsection{\label{subsec:Auxiliary-construction}Auxiliary construction: completions
of $\mathbb{A}^{3}$ into total spaces of quadric fibrations with
skewed coordinates}

To begin with, we consider a family of projective completions of $\mathbb{A}^{2}$
into quartic surfaces $Z=Z_{\alpha_{3},\alpha_{4},\alpha_{6},\alpha_{7}}$
with equations 
\begin{equation}
z_{2}^{3}z_{3}+z_{0}^{3}z_{1}+z_{2}(z_{1}^{3}+\alpha_{3}z_{1}^{2}z_{2}+\alpha_{4}z_{0}^{2}z_{1}+\alpha_{6}z_{0}z_{1}^{2}+\alpha_{7}z_{0}z_{1}z_{2})=0,\quad\alpha_{3},\alpha_{4},\alpha_{6},\alpha_{7}\in k,\label{eq:quartic-completions}
\end{equation}
in $\mathbb{P}^{3}$ with homogeneous coordinates $[z_{0}:\ldots:z_{3}]$
of type (II) in Ohta's classification \cite[Theorem 2]{Oht01}. Such
a quartic surface $Z$ contains a copy of $\mathbb{A}^{2}$ as the
open subset $\{z_{2}\neq0\}$. Blowing-up $Z$ along the curve $\{z_{1}=z_{3}=0\}\cong\mathbb{P}^{1}$
and subtracting the strict transform of $\{z_{1}=0\}$, we obtain
a cubic curve fibration $\xi=\mathrm{pr}_{\mathbb{A}^{1}}:S\to\mathbb{A}^{1}$
over $\mathbb{A}^{1}$ together with a bisection $B$ defined as the
following sub-varieties of $\mathbb{P}_{\mathbb{A}^{1}}^{2}=\mathbb{P}^{2}\times\mathbb{A}^{1}=\mathrm{Proj}_{k[x]}(k[x][z_{0},z_{1},z_{2}])$:
\begin{equation}
\begin{array}{ccl}
S & = & \{z_{2}^{3}x+z_{0}^{3}+z_{2}(z_{1}^{2}+\alpha_{3}z_{1}z_{2}+\alpha_{4}z_{0}^{2}+\alpha_{6}z_{0}z_{1}+\alpha_{7}z_{0}z_{2})=0\}\subset\mathbb{P}^{2}\times\mathbb{A}^{1},\\
B & = & \{z_{0}=0,z_{2}^{2}x+z_{1}^{2}+\alpha_{3}z_{1}z_{2}=0\}.
\end{array}\label{eq:cubic-fibration}
\end{equation}
Note that $S$ is smooth with $S\cap\{z_{2}\neq0\}$ isomorphic to
$\mathbb{A}^{2}$, that both $B$ and $S\cap\{z_{1}=0\}$ are isomorphic
to $\mathbb{A}^{1}$ and that $S\cap\{z_{0}=0\}$ is the disjoint
union of $\{z_{0}=z_{2}=0\}$ and $B$. Next we consider the closures
of the images of $\mathbb{P}_{\mathbb{A}^{1}}^{2},S,\{z_{i}=0\}$
and $S\cap\{z_{1}=0\}$ in $\mathbb{P}_{\mathbb{A}^{1}}^{3}=\mathbb{P}^{3}\times\mathbb{A}^{1}=\mathrm{Proj}_{k[x]}(k[x][u_{0},u_{1},u_{2},u_{3}])$
by the rational map 
\begin{equation}
\begin{array}{ccc}
\Phi_{B}\colon\mathbb{P}^{2}\times\mathbb{A}^{1} & \dashrightarrow & \mathbb{P}^{3}\times\mathbb{A}^{1}\\
([z_{0}:z_{1}:z_{2}],x) & \mapsto & ([z_{0}^{2}:z_{0}z_{1}:z_{0}z_{2}:z_{2}^{2}x+z_{1}^{2}+\alpha_{3}z_{1}z_{2}],x).
\end{array}\label{eq:Phi_B}
\end{equation}
defined by the $k[x]$-module $H^{0}(\mathbb{P}^{2}\times\mathbb{A}^{1},\mathcal{O}_{\mathbb{P}^{2}\times\mathbb{A}^{1}}(2)\otimes\mathcal{I}_{B})$.
It is readily checked that these varieties are described as follows:
\[
\begin{array}{ccccl}
Q^{0} & = & \overline{\Phi_{B}(\mathbb{P}_{\mathbb{A}^{1}}^{2})} & = & \{u_{2}^{2}x+u_{1}^{2}+\alpha_{3}u_{1}u_{2}-u_{0}u_{3}=0\},\\
S_{h}^{0} & = & \overline{\Phi_{B}(S)} & = & Q^{0}\cap\{u_{0}^{2}+u_{2}(\alpha_{4}u_{0}+\alpha_{6}u_{1}+\alpha_{7}u_{2})+u_{2}u_{3}=0\},\\
s^{0} & = & \overline{\Phi_{B}(\{z_{0}=0\})} & = & \{u_{0}=u_{1}=u_{2}=0\},\\
H_{i}^{0} & = & \overline{\Phi_{B}(\{z_{i}=0\})} & = & Q^{0}\cap\{u_{i}=0\}\ \ (\text{for\ }i=1,2),\\
T^{0} & = & \overline{\Phi_{B}(S\cap\{z_{1}=0\})} & = & S_{h}^{0}\cap H_{1}^{0}\cap\{u_{3}^{2}+u_{0}u_{2}x+\alpha_{4}u_{0}u_{3}+\alpha_{7}u_{2}u_{3}=0\}.
\end{array}
\]

\begin{lem}
\label{lem:linearemb} With the notation above, $Q^{0}\setminus H_{2}^{0}$
is isomorphic to $\mathbb{A}^{3}$. Moreover, there are coordinates
$(v_{1},v_{2},v_{3})$ of $Q^{0}\setminus H_{2}^{0}\cong\mathbb{A}^{3}$
such that $H_{1}^{0}\cap(Q^{0}\setminus H_{2}^{0})=\{v_{1}=0\}$,
$S_{h}^{0}\cap(Q^{0}\setminus H_{2}^{0})=\{v_{2}=0\}$, and $T^{0}\cap(Q^{0}\setminus H_{2}^{0})=\{v_{1}=v_{2}=0\}$. 
\end{lem}

\begin{proof}
By direct computation, we have 
\[
\begin{array}{ccl}
Q^{0}\setminus H_{2}^{0} & = & \{u_{2}^{2}x+u_{1}^{2}+\alpha_{3}u_{1}u_{2}-u_{0}u_{3}=0,u_{2}\neq0\}\subset\mathbb{P}^{3}\times\mathbb{A}^{1}\\
 & \cong & \{x+u_{1}^{2}+\alpha_{3}u_{1}-u_{0}u_{3}=0\}\subset\mathbb{A}_{(u_{0},u_{1},u_{3})}^{3}\times\mathbb{A}^{1}\\
 & \cong & \mathbb{A}_{(u_{0},u_{1},u_{3})}^{3},\\
H_{1}^{0}\cap(Q^{0}\setminus H_{2}^{0}) & \cong & \{u_{1}=0\}\subset\mathbb{A}_{(u_{0},u_{1},u_{3})}^{3},\\
S_{h}^{0}\cap(Q^{0}\setminus S_{h}^{0}) & \cong & \{u_{0}^{2}+\alpha_{4}u_{0}+\alpha_{6}u_{1}+\alpha_{7}+u_{3}=0\}\subset\mathbb{A}_{(u_{0},u_{1},u_{3})}^{3}.
\end{array}
\]
On the other hand, an easy verification shows that $S_{h}^{0}\cap H_{1}^{0}=s^{0}\cup T^{0}$.
Since $s^{0}\subset H_{2}^{0}$, we have $(S_{h}^{0}\cap H_{1}^{0})\cap(Q^{0}\setminus H_{2}^{0})=T^{0}\cap(Q^{0}\setminus H_{2}^{0})$,
which proves the assertion.
\end{proof}
Next we consider $\mathbb{P}^{3}\times\mathbb{A}^{1}$ as an open
subset of the projective bundle $\mathbb{F}=\mathbb{P}_{\mathbb{P}^{1}}(\mathcal{O}_{\mathbb{P}^{1}}^{\oplus3}\oplus\mathcal{O}_{\mathbb{P}^{1}}(-1))$
over $\mathbb{P}^{1}$ as follows : 
\[
\begin{array}{ccccc}
\mathbb{F} & \supset & \{x_{1}\neq0\} & \cong & \mathbb{P}_{[u_{0}:u_{1}:u_{2}:u_{3}]}^{3}\times\mathbb{A}_{(x)}^{1}\\
 &  & ([w_{0}:w_{1}:w_{2}:w_{3}],[x_{0}:x_{1}]) & \mapsto & ([w_{0}:w_{1}:w_{2}:w_{3}x_{1}^{-1}],x_{0}x_{1}^{-1}).
\end{array}
\]
The closures of the images of $Q^{0},S_{h}^{0},s^{0},H_{i}^{0}$ and
$T^{0}$ in $\mathbb{F}$ are then described as follows: 
\[
\begin{array}{ccccl}
Q & = & \overline{Q^{0}}^{\mathbb{F}} & = & \{w_{2}^{2}x_{0}+x_{1}(w_{1}^{2}+\alpha_{3}w_{1}w_{2})-w_{0}w_{3}=0\},\\
S_{h} & = & \overline{S_{h}^{0}}^{\mathbb{F}} & = & Q\cap\{x_{1}(w_{0}^{2}+w_{2}(\alpha_{4}w_{0}+\alpha_{6}w_{1}+\alpha_{7}w_{2}))+w_{2}w_{3}=0\},\\
s & = & \overline{s^{0}}^{\mathbb{F}} & = & \{w_{0}=w_{1}=w_{2}=0\},\\
H_{i} & = & \overline{H_{i}^{0}}^{\mathbb{F}} & = & Q\cap\{w_{i}=0\}\ \ (\text{for\ }i=1,2),\\
T & = & \overline{T^{0}}^{\mathbb{F}} & = & S_{h}\cap H_{1}\cap\{w_{3}^{2}+x_{1}(w_{0}w_{2}x_{0}+\alpha_{4}w_{0}w_{3}+\alpha_{7}w_{2}w_{3})=0\}.
\end{array}
\]
The restriction to $Q$ of the $\mathbb{P}^{3}$-bundle structure
$p\colon\mathbb{F}\to\mathbb{P}^{1}$ is a quadric fibration $q:Q\to\mathbb{P}^{1}$.
We let $\mathcal{O}_{Q}(1)=\mathcal{O}_{\mathbb{F}}(1)|_{Q}$ and
$S_{f}=q^{-1}([1:0])$. Then $H_{i}\in|\mathcal{O}_{Q}(1)|$ for $i=1,2$,
$S_{h}\sim2H_{1}+S_{f}$ and $-K_{Q}\sim2H_{1}+2S_{f}$. The following
lemma summarizes properties of the projective varieties above. 
\begin{lem}
\label{lem:Quadric-Setup-Properties}With the notation above, the
following hold: 

a) The variety $Q$ is a smooth weak Fano threefold of type (2,3,2)
in \cite[Theorem 2.3]{Tak22}. In particular, $(-K_{Q})^{3}=40$ and
$s$ is the unique $(-K_{Q})$-trivial curve. 

b) The curve $T$ is a smooth rational trisection of $q:Q\to\mathbb{P}^{1}$.

c) The curve $S_{h}\cap H_{1}$ is the scheme theoretic disjoint union
of $s$ and $T$.

d) The surface $H_{1}$ is smooth and the restriction $q|_{H_{1}}:H_{1}\to\mathbb{P}^{1}$
is a conic bundle with a unique singular fiber $H_{1}\cap q^{-1}([0:1])=\{x_{0}=w_{0}=w_{1}=0\}\cup\{x_{0}=w_{1}=w_{3}=0\}$.
Moreover, the $q$-trisection $T$ intersects with $\{x_{0}=w_{0}=w_{1}=0\}$
(resp.\ $\{x_{0}=w_{1}=w_{3}=0\}$) once (resp.\ twice). 

e) The surface $H_{2}$ is smooth and the restriction $q|_{H_{2}}:H_{2}\to\mathbb{P}^{1}$
is a conic bundle with a unique singular fiber $H_{2}\cap q^{-1}([1:0])=H_{2}\cap S_{f}$.
The support of $S_{h}\cap H_{2}$ is the union of $s$ and curves
in $S_{f}$. Moreover, $H_{1}\cap H_{2}$ is the union of two disjoint
section $s\sqcup\{w_{1}=w_{2}=w_{3}=0\}$. 

f) The singular locus of the surface $S_{h}$ equals $\{x_{1}=w_{0}=w_{2}=w_{3}=0\}$
and the general fibers of $q|_{S_{h}}:S_{h}\to\mathbb{P}^{1}$ are
elliptic curves.
\end{lem}

\begin{proof}
The smoothness of $Q$ is checked locally by the Jacobian criterion.
Since $Q$ is a member of the complete linear system $|\mathcal{O}_{\mathbb{F}}(2)\otimes p^{*}\mathcal{O}_{\mathbb{P}^{1}}(1)|$,
Assertion a) holds. Recall that $S\cap\{z_{1}=0\}\subset\mathbb{P}_{\mathbb{A}^{1}}^{2}$
is a rational trisection. Since there is a birational morphism $\mathbb{A}^{1}\cong S\cap\{z_{1}=0\}\rightarrow T$
over the base curve, $T$ is also rational trisection, whose smoothness
follows again from the Jacobian criterion, which proves b). One can
check locally that $S_{h}\cap H_{1}$ contains no $q$-exceptional
curve. By construction, we have $S_{h}\cap H_{1}\supset s\cup T$
as a set and $s\cap T=\{w_{0}=w_{1}=w_{2}=w_{3}=0\}=\emptyset$. Since
$(S_{h}\cdot H_{1}\cdot S_{f})_{Q}=(2\mathcal{O}_{Q}(1)+S_{f}\cdot\mathcal{O}_{Q}(1)\cdot S_{f})_{Q}=4$
and $s\cup T$ is a reducible $4$-section of $q$, Assertion c) follows.
The remaining assertions follow from straightforward computations. 
\end{proof}

\subsubsection{Construction of a del Pezzo completion $(X,B_{h}\cup B_{f})$ of
degree $6$ and proof of Theorem B }
\begin{prop}
\label{prop:Blowup-link}Let $q:Q\to\mathbb{P}^{1}$ be the quadric
fibration constructed in subsection \ref{subsec:Auxiliary-construction}
above and let $f:Y\to Q$ be the blow-up along $T$. Then the following
hold:

a) The variety $Y$ is a weak Fano threefold of Picard rank three. 

b) There exists a Sarkisov link \[\xymatrix@!C{ & Y\ar[dl]_-{f}\ar@{.>}[r]^-{\chi} & Y^{+}\ar[dr]^-{g}\\ Q\ar[d]_-{q} & & & X\ar[d]^-{\pi}\\ \mathbb{P}^{1}\ar@{=}[rrr] & & & \mathbb{P}^{1}}\]
where $\chi$ is the Atiyah flop of the strict transform $l$ in $Y$
of the curve $\{x_{0}=w_{1}=w_{3}=0\}\subset Q$, $\pi:X\to\mathbb{P}^{1}$
is a del Pezzo fibration of degree $6$ and $g$ is the blow-up along
a section of $\pi$ with exceptional divisor $(H_{1})_{Y^{+}}$.

c) The flopped curve $l^{+}\subset Y^{+}$ of $\chi$ is contained
in $(H_{2})_{Y^{+}}$ but not in $(S_{h}\cup H_{1})_{Y^{+}}$ and
the complement $U_{H_{2}}\subset(H_{2})_{Y^{+}}$ of $(H_{2})_{Y^{+}}\cap(S_{f}\cup S_{h}\cup H_{1})_{Y^{+}}$
is isomorphic to $\mathbb{A}^{2}$.
\end{prop}

\begin{proof}
First we show that $-K_{Y}$ is nef. Suppose that a $(-K_{Y})$-negative
curve $r\subset Y$ exists. Then $r\subset(S_{h})_{Y}$ since $-K_{Y}\sim f^{*}S_{f}+(S_{h})_{Y}$
and $f^{*}S_{f}$ is nef. On the other hand, we have $-K_{Y}\sim f^{*}H_{1}+2f^{*}S_{f}+(H_{1})_{Y}$.
The base locus of $|f^{*}H_{1}|$ equals $s_{Y}$ because the base
locus of $|H_{1}|$ is $H_{1}\cap H_{2}\cap\{w_{0}=0\}=s$ and $s\cap T=\emptyset$.
Thus $r\subset s_{Y}\cup(H_{1})_{Y}=(H_{1})_{Y}$. Since $r\subset(S_{h})_{Y}\cap(H_{1})_{Y}=s_{Y}$,
we conclude that $r=s_{Y}$. However, since $s\cap T=\emptyset$,
we have $(-K_{Y}\cdot s_{Y})_{Y}=(-K_{Q}\cdot s)_{Q}=0$, a contradiction.
Hence $-K_{Y}$ is nef. On the other hand, we have $(-K_{Q}\cdot T)_{Q}=(-K_{Q}\cdot T+s)_{Q}=(-K_{Q}\cdot S_{h}\cdot H_{1})_{Q}=8$.
Hence $(-K_{Y})^{3}=(-K_{Q})^{3}-2\{(-K_{Q}\cdot T)_{Q}+1\}=22$,
and $-K_{Y}$ is big, which proves Assertion a). 

The existence of the claimed Sarkisov link with $g$-exceptional divisor
$(H_{1})_{Y^{+}}$ follows from \cite[Proposition 3.5 (2)]{Fuk17}
and \cite[Proposition 2.1 (2)]{Fuk18}. The fact that $\chi$ is the
Atiyah flop of $l$ can be verified as follows: let $r\subset Y$
be a curve flopped by $\chi$ and set $t=(q\circ f)(r)$. If $r$
is $f$-exceptional, then $(-K_{Y}\cdot r)=1$, a contradiction. Since
$-K_{Y}\sim f^{*}H_{1}+2f^{*}S_{f}+(H_{1})_{Y}$, we have $0=(-K_{Y}\cdot r)_{Y}=(f^{*}H_{1}+(H_{1})_{Y}\cdot r)_{Y}=(H_{1}\cdot f_{*}r)_{Q}+((H_{1})_{Y}\cdot r)_{Y}$.
Since $H_{1}$ is $q$-ample, we obtain $((H_{1})_{Y}\cdot r)_{Y}<0$.
Hence $r\subset(H_{1})_{Y}$. Since $T\subset H_{1}$, we have $(H_{1})_{Y}\cong H_{1}$.
If $H_{1}\cap q^{-1}(t)\cong(H_{1})_{Y}\cap(q\circ f)^{-1}(t)$ is
smooth, then $(-K_{Y}\cdot r)_{Y}=(-K_{Y}\cdot(H_{1})_{Y}\cdot(q\circ f)^{-1}(t))_{Y}=1$,
a contradiction. Hence $H_{1}\cap q^{-1}(t)$ is singular, which implies
that $t=[0:1]$. Since $(-K_{Y}\cdot r)_{Y}=0$ and $(-K_{Q}\cdot f_{*}r)_{Q}=2$,
$f_{*}r$ is the irreducible component of $q^{-1}(t)$ intersecting
with $T$ twice. Lemma \ref{lem:Quadric-Setup-Properties} d) now
implies that $r=l$. On the other hand, since $l\subset(H_{1})_{Y}$
is a $(-1)$-curve and $(l\cdot(H_{1})_{Y})_{Y}=(l\cdot f^{*}H_{1})_{Y}-2=-1$,
we deduce from the normal bundle sequence 
\[
\xymatrix{0\ar[r] & {N_{l}(H_{1})_{Y}}\ar[r] & {N_{l}Y}\ar[r] & {(N_{(H_{1})_{Y}}Y)|_{l}}\ar[r] & 0}
\]
that $N_{l}Y\cong\mathcal{O}_{\mathbb{P}^{1}}(-1)^{\oplus2}$, hence
that $\chi$ is the Atiyah flop of $l$. 

To prove Assertion c), we first observe that $S_{h}\cap f_{*}l$ is
$0$-dimensional since $S_{h}\cap H_{1}=s\sqcup T$. Then $(S_{h})_{Y}\cap l=\emptyset$
since $((S_{h})_{Y}\cdot l)=(-K_{Y}-(f^{*}S_{f})\cdot l)=0$ and hence,
$l^{+}$ is not contained in $(S_{h})_{Y^{+}}$. On the other hand,
let $\eta\colon\widetilde{Y}\to Y$ be the blow-up along $l$ and
$G$ be the $\eta$-exceptional divisor. Then the contraction of the
other ruling of $G\cong\mathbb{P}^{1}\times\mathbb{P}^{1}$ gives
the morphism $\eta^{+}\colon\widetilde{Y}\to Y^{+}$. Since $N_{l}Y\cong\mathcal{O}_{\mathbb{P}^{1}}(-1)^{\oplus2}$,
we have $G^{3}=2$ and 
\[
((H_{1})_{\widetilde{Y}}|_{G})^{2}=((H_{1})_{\widetilde{Y}}^{2}\cdot G)_{\widetilde{Y}}=((\eta^{*}(H_{1})_{Y}-G)^{2}\cdot G)_{\widetilde{Y}}=-2(\eta^{*}(H_{1})_{Y}\cdot G^{2})_{\widetilde{Y}}+G^{3}=2((H_{1})_{Y}\cdot l)_{Y}+2=0,
\]
which implies that $(H_{1})_{\widetilde{Y}}\cap G$ is contracted
by $\eta^{+}$. Hence $l^{+}$ is not contained in $(H_{1})_{Y^{+}}$.
Since $s\cap T=\emptyset$ and the support of $S_{h}\cap H_{2}$ is
contained in $s\cup S_{f}$, the support of $T\cap H_{2}$ is contained
in $S_{f}$. An easy computation shows that $(f_{*}l)\cap H_{2}=\{x_{0}=w_{1}=w_{2}=w_{3}=0\}\in H_{1}$.
Hence the morphism $(H_{2})_{Y^{+}}\setminus((H_{2})_{Y^{+}}\cap(S_{f})_{Y^{+}})\to H_{2}\setminus(H_{2}\cap S_{f})$
is the blow up at a point whose exceptional divisor is $l^{+}$. Lemma
\ref{lem:Quadric-Setup-Properties} e) then implies that $(H_{2})_{Y^{+}}\setminus((H_{2})_{Y^{+}}\cap(S_{f})_{Y^{+}})$
is isomorphic to the blow-up of $\mathbb{P}^{1}\times\mathbb{A}^{1}$
at a point, say $x$, and that it contains $U_{H_{2}}$ as the complement
of the strict transform of two disjoint sections one of which passes
through $x$. Therefore, $U_{H_{2}}\cong\mathbb{A}^{2}$. 
\end{proof}
\begin{rem}
We note that $-K_{X}$ is not nef since $(-K_{X}\cdot s_{0})=(-K_{Q}\cdot T)-9=-1$
by \cite[Proposition 3.5 (3)]{Fuk17}. In particular, $\pi\colon X\to\mathbb{P}^{1}$
is distinct from the del Pezzo fibration of degree $6$ constructed
in \cite[Theorem 6.6]{Pro16}.
\end{rem}

\begin{thm}
\label{thm:A3} The threefold $X$ is a completion of $\mathbb{A}^{3}$
with boundary divisor $B=B_{h}\cup B_{f}=(S_{h}\cup S_{f})_{X}$ and
the restriction $\bar{\pi}:B_{h}\setminus(B_{h}\cap B_{f})\to\mathbb{P}^{1}\setminus\pi(B_{f})$
is isomorphic the cubic curve fibration $\xi:S\to\mathbb{A}^{1}$
in (\ref{eq:cubic-fibration}). 
\end{thm}

\begin{proof}
Since $-K_{Y}\sim f^{*}S_{f}+(S_{h})_{Y}$, we have $-K_{X}\sim(S_{f})_{X}+(S_{h})_{X}$.
Since the classes of $-K_{X}$ and $(S_{f})_{X}$ form a $\mathbb{Z}$-basis
of the Picard group of $X$ and $-K_{X}$ is $\pi$-ample, the complement
$X'=X\setminus(S_{f}\cup S_{h})_{X}$ is a smooth affine threefold
whose coordinate ring, say $A$, is a factorial domain whose group
of invertible elements equals $k^{*}$. Since $(S_{h})_{Y^{+}}$ intersects
the $g$-exceptional divisor $(H_{1})_{Y^{+}}$ along $s_{Y^{+}}$,
we have $X'\cong Y^{+}\setminus(S_{f}\cup S_{h}\cup H_{1})_{Y^{+}}$.
In particular, $X'$ contains the Zariski open subset $U^{+}=Y^{+}\setminus(S_{f}\cup S_{h}\cup H_{1}\cup H_{2})_{Y^{+}}$.
By Proposition \ref{prop:Blowup-link} c), the complement $X'\setminus U^{+}=U_{H_{2}}$
is isomorphic to $\mathbb{A}^{2}$. On the other hand, by Lemma \ref{lem:linearemb},
$U=Y\setminus(S_{f}\cup S_{h}\cup H_{1}\cup H_{2})_{Y}$ is isomorphic
to the sub-variety in the blow-up of $\mathbb{A}^{3}$ along a linearly
embedded $\mathbb{A}^{1}$ with complement equal to the strict transform
of two linearly embedded affine planes $\mathbb{A}^{2}$ containing
$\mathbb{A}^{1}$, hence is isomorphic to $\mathbb{A}^{2}\times(\mathbb{A}^{1}\setminus\{0\})$.
Since $l$ (resp. $l^{+}$) is contained in $(H_{1})_{Y}$ (resp.
$(H_{2})_{Y^{+}}$), $\chi$ induces an isomorphism between $U$ and
$U^{+}$. Since $A$ is factorial, there exists an element $s\in A$
such that $A[s^{-1}]$ is the coordinate ring of $U^{+}$, which is
isomorphic to $k[z,z^{-1}][y_{1},y_{2}]$ for some algebraically independent
elements $y_{1},y_{2},z$. By the same arguments as in \cite[Section 2.4]{Miy84},
we deduce that the subring $R=A\cap k[z,z^{-1}]\subset k[z,z^{-1}][y_{1},y_{2}]$
is isomorphic to a polynomial ring in one variable over $k$. Furthermore,
$R$ contains $s$ and $R[s^{-1}]=k[z,z^{-1}]$, implying that $R=k[s]$.
Let $p'\colon X'\to\mathbb{A}^{1}$ be the morphism induced by the
inclusion $k[s]=R\subset A$. By construction, the restriction of
$p'$ over $\mathbb{A}^{1}\setminus\{0\}$ is isomorphic to the trivial
$\mathbb{A}^{2}$-bundle whereas the fiber over $\{s=0\}$ is equal
to $\mathrm{Spec}(A/(s))=U_{H_{2}}\cong\mathbb{A}^{2}$. By \cite[Theorem 1]{Sat83},
$p'\colon X'\to\mathbb{A}^{1}$ is thus a locally trivial $\mathbb{A}^{2}$-bundle
in the Zariski topology, hence is isomorphic to the trivial one by
\cite{BCW76}, showing that $X'\cong\mathbb{A}^{3}$. For the second
assertion, we note that since the birational map $\Phi_{B}$ in (\ref{eq:Phi_B})
is the composition of the relative Veronese embedding $\mathbb{P}_{\mathbb{A}^{1}}^{2}\to\mathbb{P}_{\mathbb{A}^{1}}^{5}$
with the projection $\mathbb{P}_{\mathbb{A}^{1}}^{5}\dashrightarrow\mathbb{P}_{\mathbb{A}^{1}}^{3}$
from the relative projective line $\mathbb{P}_{\mathbb{A}^{1}}^{1}$
which is the linear hull of $B$, it induces an isomorphism $S\cong S_{h}^{0}$
of varieties over $\mathbb{A}^{1}$. On the other hand, since away
from the fiber over $[1:0]$, $T$ is contained in the smooth locus
of $S_{h}$, $l$ is disjoint from $(S_{h})_{Y}$, and $(S_{h})_{Y^{+}}$
intersects with $(H_{1})_{Y^{+}}$ along $s_{Y^{+}}$ transversely,
the birational map $Q\dashrightarrow X$ also induces an isomorphism
$S_{h}^{0}\cong B_{h}\setminus(B_{h}\cap B_{f})$ of varieties over
$\mathbb{A}^{1}$. 
\end{proof}
\begin{rem}
The affine open subset $X\setminus(S_{h}\cup S_{f})_{X}\cong\mathbb{A}^{3}$
of $X$ is not preserved by the inverse of the flop $\chi$. Indeed,
since $U_{H_{2}}\cap l_{+}\cong\mathbb{A}^{1}$, we have $Y\setminus(S_{f}\cup S_{h}\cup H_{1})_{Y}\cong(\mathbb{A}^{2}\setminus\{0\})\times\mathbb{A}^{1}$,
which is not affine. 
\end{rem}

By construction, for every choice of parameters $\alpha_{3},\alpha_{4},\alpha_{6},\alpha_{7}\in k$
for the initial completion $Z=Z_{\alpha_{3},\alpha_{4},\alpha_{6},\alpha_{7}}$
of $\mathbb{A}^{2}$ in \eqref{eq:quartic-completions}, Theorem \ref{thm:A3}
renders a completion of $\mathbb{A}^{3}$ into the total space of
a del Pezzo fibration $\pi:X\to\mathbb{P}^{1}$ of degree $6$ for
which the closed fibers of the restriction $\bar{\pi}:B_{h}\setminus(B_{h}\cap B_{f})\to\mathbb{P}^{1}\setminus\pi(B_{f})=\mathbb{A}^{1}$
of $\pi$ are isomorphic to those of the cubic curve fibration $\xi:S\to\mathbb{A}^{1}$
in (\ref{eq:cubic-fibration}). The isomorphism classes of the closed
fibers of $\xi:S\to\mathbb{A}^{1}$ can in turn be easily determined
in terms of the parameters $\alpha_{3},\alpha_{4},\alpha_{6},\alpha_{7}\in k$.
For instance, for $(\alpha_{3},\alpha_{4},\alpha_{6},\alpha_{7})=(0,0,0,0)$,
we have
\[
S=\{z_{2}^{3}x+z_{0}^{3}+z_{2}z_{1}^{2}=0\}\subset\mathbb{P}_{[z_{0}:z_{1}:z_{2}]}^{2}\times\mathbb{A}_{(x)}^{1}
\]
and hence, every closed fiber of $\xi=\mathrm{pr}_{\mathbb{A}^{1}}:S\to\mathbb{A}^{1}$
other than that over $0$ is an elliptic curve with $j$-invariant
$0$. On the other hand, for $(\alpha_{3},\alpha_{4},\alpha_{6},\alpha_{7})=(0,0,0,1)$,
we have 
\[
S=\{z_{2}^{3}x+z_{0}^{3}+z_{2}(z_{1}^{2}+z_{0}z_{2})=0\}\subset\mathbb{P}_{[z_{0}:z_{1}:z_{2}]}^{2}\times\mathbb{A}_{(x)}^{1}
\]
and a direct computation yields that the closed fibers $S_{x}$ of
$\mathrm{pr}_{\mathbb{A}^{1}}:S\to\mathbb{A}^{1}$ other than those
over $x=\pm2i/(3\sqrt{3})$, where $i=\sqrt{-1}\in k$, are elliptic
curves with $j$-invariant $j(S_{x})=1728\frac{4}{4+27x^{2}}$. These
two choices show in particular that every isomorphism type of elliptic
curve can be realized as a closed fiber of $\bar{\pi}:B_{h}\setminus(B_{h}\cap B_{f})\to\mathbb{P}^{1}\setminus\pi(B_{f})=\mathbb{A}^{1}$
for a suitable choice of parameters $\alpha_{3},\alpha_{4},\alpha_{6},\alpha_{7}\in k$,
which completes the proof of Theorem B in the case $d=6$. 

\bibliographystyle{amsplain}

\end{document}